\documentclass[a4paper,reqno, oneside]{amsart}
\usepackage[english]{babel}
\usepackage{amsmath, amssymb, amsthm, amscd}
\usepackage{enumerate}
\usepackage{palatino}
\usepackage{mathpazo}
\usepackage{paralist}
\usepackage[a4paper]{geometry}
\usepackage{hyperref}

\usepackage{tikz} 
\usepackage{tikz-cd}

\DeclareMathOperator{\id}{id}
\DeclareMathOperator{\im}{im}

\newcommand{\NN}{\mathbb{N}}
\newcommand{\ZZ}{\mathbb{Z}}
\newcommand{\RR}{\mathbb{R}}
\newcommand{\CC}{\mathbb{C}}

\newcommand{\Ad}{\mathrm{Ad}}
\newcommand{\Cut}{\mathrm{Cut}}
\newcommand{\Cuttil}{\widetilde{\Cut}}

\newcommand{\Isom}{\mathrm{Isom}}

\newcommand{\D}{\mathcal{D}}

\renewcommand{\gg}{\mathfrak{g}}
\newcommand{\hh}{\mathfrak{h}}
\newcommand{\mm}{\mathfrak{m}}
\renewcommand{\aa}{\mathfrak{a}}
\newcommand{\kk}{\mathfrak{k}}

\DeclareMathOperator{\GC}{\mathsf{GC}}
\DeclareMathOperator{\TC}{\mathsf{TC}}
\DeclareMathOperator{\cat}{\mathsf{cat}}
\DeclareMathOperator{\secat}{\mathsf{secat}}

\DeclareMathOperator{\Exp}{Exp}

\renewcommand{\setminus}{\smallsetminus}

\theoremstyle{plain}
\newtheorem{theorem}{Theorem}[section]
\newtheorem{prop}[theorem]{Proposition}
\newtheorem{lemma}[theorem]{Lemma}
\newtheorem{cor}[theorem]{Corollary}  

\newtheorem*{thm}{Theorem}

\theoremstyle{definition}
\newtheorem{definition}[theorem]{Definition}

\theoremstyle{remark}
\newtheorem{remark}[theorem]{Remark}
\newtheorem{example}[theorem]{Example}

\numberwithin{equation}{section}

\begin{document}
\setlength{\parindent}{0.cm}

\title{Geodesic complexity via fibered decompositions of cut loci}

\author{Stephan Mescher}
\address{Institut f\"ur Mathematik \\ Martin-Luther-Universit\"at Halle-Wittenberg \\ Theodor-Lieser-Strasse 5 \\ 06120 Halle (Saale) \\ Germany}
\email{stephan.mescher@mathematik.uni-halle.de}

\author{Maximilian Stegemeyer}
\address{Max Planck Institute for Mathematics in the Sciences, Inselstrasse 22, 04103 Leipzig, Germany} 
\address{Mathematisches Institut, Universit\"at Leipzig, Augustusplatz 10, 04109 Leipzig, Germany}
\email{maximilian.stegemeyer@mis.mpg.de}
\date{\today}
\maketitle

\begin{abstract}
The geodesic complexity of a Riemannian manifold is a numerical isometry invariant that is determined by the structure of its cut loci. 
In this article we study decompositions of cut loci over whose components the tangent cut loci fiber in a convenient way. We establish a new upper bound for geodesic complexity in terms of such decompositions. As an application, we obtain estimates for the geodesic complexity of certain classes of homogeneous manifolds. In particular, we compute the geodesic complexity of complex and quaternionic projective spaces with their standard symmetric metrics. 
\end{abstract}

\tableofcontents

\section{Introduction}

The geodesic complexity of a complete Riemannian manifold is an integer-valued isometry invariant. It is given as a geometric analogue of the notion of topological complexity as introduced by M. Farber in \cite{farber:2003}.
Geodesic complexity was originally defined by D. Recio-Mitter in the more general framework of metric spaces in \cite{reciomitter:2021}.
Given a complete Riemannian manifold $(M,g)$ we denote its space of length-minimizing geodesic segments by $GM$, seen as a subspace of the path space $C^0([0,1],M)$ with the compact-open topology.
Consider the endpoint evaluation map $$p \colon GM\to M\times M, \qquad p(\gamma) = (\gamma(0),\gamma(1)).$$ 

The geodesic complexity of $(M,g)$, denoted by $\GC(M,g)$, is defined as the smallest integer $k$ for which there exists a decomposition of $M\times M$ into locally compact subsets $A_1,\ldots, A_k$ with each $A_i$ admitting a continuous local section of $p$.

In the same way that topological complexity is motivated by a topological abstraction of the motion planning problem from robotics, geodesic complexity is motivated by an abstract notion of \emph{efficient} motion planning. Sections of $p$ can be seen as geodesic motion planners since they assign to a pair of points $(p,q)\in M\times M$ a length-minimizing path connecting these two points.

As noted by Recio-Mitter in \cite[p. 144]{reciomitter:2021}, the main problem in determining the geodesic complexity of $(M,g)$ lies in understanding the geodesic motion planning problem on its total cut locus. The latter is defined as
$$  \Cut(M) = \{ (p,q)\in M\times M\,|\, q\in \Cut_p(M)\} \subseteq M\times M ,     $$
where $\Cut_p(M) \subset M$ denotes the cut locus of $p \in M$ with respect to the given metric $g$. 
In this article we introduce the notion of a \textit{fibered decomposition of the total cut locus}. If such a decomposition exists, it gives rise to a new upper bound for the geodesic complexity of $M$.
The main applications of this upper bound are estimates for the geodesic complexity of certain homogeneous Riemannian manifolds. Similar situations were already studied by the authors in \cite{mescher:2021}.
The upper bounds in the present article are however independent of the ones given in \cite{mescher:2021}. Various estimates from that article can be improved using the new results. \medskip

Let $(M,g)$ be a complete Riemannian manifold and consider the extended exponential map 
$$\Exp\colon TM\to M\times M, \qquad\Exp(v) = (\mathrm{pr}(v), \exp_{\mathrm{pr}(v)}(v)),$$ where $\mathrm{pr}:TM \to M$ denotes the bundle projection.
We say that the total cut locus \emph{admits a fibered decomposition} if there is a decomposition of $\Cut(M)$ into locally compact subsets $A_1,\ldots, A_k$, such that the restriction
$$  \pi_i:=  \Exp|_{\widetilde{A}_i} \colon \widetilde{A}_i \to A_i,     $$
where $\widetilde{A}_i = \Exp^{-1}(A_i)\cap \Cuttil(M)$, is a fibration for each $i \in \{1,\ldots,k\}$. 
Here, $\Cuttil(M)$ denotes the \textit{total tangent cut locus} which will be defined below. We will establish the following result.
 \begin{thm}[Theorem \ref{theorem_fibered_decomp}]
 Let $(M,g)$ be a complete Riemannian manifold.
 If the total cut locus $\Cut(M)$ admits a fibered decomposition $A_1,\ldots, A_k$ with fibrations $\pi_i\colon \widetilde{A}_i\to A_i$ for $i\in\{1,\ldots,k\}$, then the geodesic complexity of $M$ can be estimated by
 $$ \GC(M,g) \leq \sum_{i=1}^k \secat(\pi_i \colon \widetilde{A}_i \to A_i)  + 1.      $$
 \end{thm}
Here, $\secat$ denotes the sectional category of a fibration, which was introduced by A. Schwarz in \cite{schwarz:1966} as the genus of a fiber space.

Evidently, this raises the question whether there are any interesting cases of Riemannian manifolds whose total cut loci admit fibered decompositions.
For a homogeneous Riemannian manifold we will establish a tangible criterion on the cut locus of a single point implying that its total cut locus admits a fibered decomposition.
We will further show that each irreducible compact simply connected symmetric space satisfies this condition, providing a large class of examples whose total cut loci admit fibered decompositions.
By applying this result, we are able to compute the geodesic complexity of complex and quaternionic projective spaces with respect to their standard symmetric metrics.

\begin{thm}[Theorem \ref{theorem_projective}]
Let $M =  \CC P^n$ or $\mathbb{H}P^n$ equipped with its standard or Fubini-Study metric $g_{\mathrm{sym}}$, where $n \in \NN$.
Then its geodesic complexity satisfies
$$  \GC(M,g_{\mathrm{sym}}) = 2n + 1 .      $$
In particular, the geodesic complexity of $(M,g_{\mathrm{sym}})$ equals the topological complexity of $M$.
\end{thm}

Moreover, using results by V. Ozols from \cite{ozols:1974}, we study the total cut locus of three-di\-men\-sio\-nal lens spaces with metrics of constant sectional curvature. We show that lens spaces of the form $L(p;1)$, where $p\geq 3$, are further examples of homogeneous manifolds whose total cut loci admit fibered decompositions. As these spaces are not globally symmetric, this shows that fibered decompositions are not exclusively obtained in the globally symmetric case. A detailed analysis of the fibrations involved in the fibered decompositions of $\Cut(L(p;1))$ shows that
$$   6\leq \GC(L(p;1),g) \leq 7 ,$$
see Theorem \ref{theorem_lens}, where $g$ is a metric of constant sectional curvature. \medskip 

This manuscript is organized as follows.
In Section \ref{sec_basics} we review the definitions of the total cut locus and of geodesic complexity and note some basic properties of these objects.

The central notion of a fibered decomposition of the total cut locus is introduced in Section \ref{sec_fibered_decomp}. 
In that section we also prove the above mentioned upper bound on geodesic complexity and study a criterion for the existence of a fibered decomposition.

Symmetric spaces are studied in Section \ref{sec_symm}. After recalling some properties of root systems and related notions we prove that the total cut loci of irreducible compact simply connected symmetric spaces admit fibered decompositions and derive an upper bound on geodesic complexity.
This will be applied to the examples of complex and quaternionic projective spaces and a particular complex Grassmannian. 
 
Finally, in Section \ref{sec_lens} we discuss the total cut loci of three-dimensional lens spaces and study a fibered decomposition to derive an upper bound on the geodesic complexity of these spaces.

\section{Geodesic complexity and the total cut locus} \label{sec_basics}

In this section we quickly introduce the basic notions of geodesic complexity and of the total cut locus.
For more properties of geodesic complexity and of the relation between cut loci and geodesic complexity we refer to \cite{reciomitter:2021} and \cite{mescher:2021}.

Under a \emph{locally compact decomposition} of a topological space $X$ we understand a cover
$ A_1,\ldots, A_k  $ of $X$ such that the $A_i$ are pairwise disjoint and each $A_i$, $i\in\{1,\ldots k\}$, is a locally compact subspace of $X$.
As usual we equip the path space $C^0(I,M)$ with the compact-open topology, where $I=[0,1]$ is the unit interval.
For a Riemannian manifold $(M,g)$ we let $GM \subseteq C^0(I,M)$ be the space of length-minimizing paths in $M$ equipped with the subspace topology of $C^0(I,M)$, i.e.
$$  GM = \{\gamma\in C^0(I,M)\,|\, \gamma \text{ is a length-minimizing geodesic in }M\} .  $$
\begin{definition} \label{def_gc}
Let $(M,g)$ be a complete Riemannian manifold and let 
$$ p : GM\to M\times M, \qquad p (\gamma)=(\gamma(0),\gamma(1)).$$ 
\begin{enumerate}
    \item A local section of $p $ is called a \textit{geodesic motion planner}.
\item Let $B\subseteq M\times M$ be a subset.
The \textit{subspace geodesic complexity of $B$ in $(M,g)$} is defined to be the smallest integer $k$ for which there is a locally compact decomposition $A_1,\ldots, A_k$ of $B$ with the following property: for each $i \in \{1,\dots,k\}$ there exists a continuous geodesic motion planner $A_i \to GM$. 
The subspace geodesic complexity of $B$ in $(M,g)$ is denoted by $\GC_{(M,g)}(B)$. If no such $k$ exists, we put $\GC_{(M,g)}(B):=+\infty$.
\item The \emph{geodesic complexity of $(M,g)$} is defined to be the subspace geodesic complexity of $M\times M$ itself and is denoted by $\GC(M,g)$, i.e. $\GC(M,g) = \GC_{(M,g)}(M\times M)$. 
\end{enumerate}
\end{definition}
\begin{remark}
\begin{enumerate}
    \item 
By the definition of topological complexity via locally compact decompositions, see \cite[Section 4.3]{farber:2008}, it is clear that the geodesic complexity of a Riemannian manifold $(M,g)$ is bounded from below by the topological complexity $\TC(M)$ of $M$.
\item If the metric under consideration is apparent, then we will drop the metric from the notation and simply write $\GC(M):=\GC(M,g)$ or $\GC_M(B) :=\GC_{(M,g)}(B)$.
\item Geodesic complexity was introduced by D. Recio-Mitter in \cite{reciomitter:2021} for more general geodesic spaces, i.e. metric spaces in which any two points are connected by a length-minimizing path. Since every complete Riemannian manifold is a geodesic space, our definition is nothing but a particular case of Recio-Mitter's definition.
Note however that our definition of geodesic complexity differs from the one in \cite{reciomitter:2021} by one. More precisely, while in \cite{reciomitter:2021} a geodesic space of geodesic complexity $k\in \NN$ is decomposed into at least $k+1$ locally compact subsets admitting geodesic motion planners, our definition requires the existence of a decomposition into $k$ subsets having this property. 
\end{enumerate}
\end{remark}
As pointed out by Recio-Mitter in \cite{reciomitter:2021} the geodesic complexity of a Riemannian manifold $(M,g)$ crucially depends on the cut loci of $M$.
We next recall the notion of the cut locus of a point as well as the total cut locus and the total tangent cut locus of a Riemannian manifold. The latter two notions were introduced by Recio-Mitter in \cite{reciomitter:2021}.
\begin{definition}
Let $(M,g)$ be a complete Riemannian manifold and let $p\in M$.
\begin{enumerate}
    \item Let $\gamma\colon [0,\infty)\to M$ be a unit-speed geodesic with $\gamma(0) = p$.
    We say that the \textit{cut time of} $\gamma$ is 
    $$   t_{\text{cut}}(\gamma) = \sup \{ t > 0\,|\, \gamma|_{[0,t]} \,\,\text{is minimal}\} .    $$
    In case that $t_{\text{cut}}(\gamma)<\infty$ we say that $\gamma(t_{\text{cut}}(\gamma))$ is a \textit{cut point of} $p$ \textit{along $\gamma$} and that $t_{\text{cut}}(\gamma)\Dot{\gamma}(0)\in T_pM$ is a \textit{tangent cut point} of $p$.
    \item The set of tangent cut points of $p$ is called the \textit{tangent cut locus of} $p$ and is denoted by $\Cuttil_p(M)\subset T_pM$.
    The set of cut points of $p$ is called the \textit{cut locus of} $p$ and is denoted by $\Cut_p(M)$.
    \item The \textit{total tangent cut locus of} $M$ is given by
    $$ \Cuttil(M) = \bigcup_{p\in M} \Cuttil_p(M) \subseteq TM  .  $$
    The \textit{total cut locus of} $M$ is defined as
    $$  \Cut(M) = \bigcup_{p\in M}(\{p\}\times \Cut_p(M)) \subseteq M\times M .    $$
\end{enumerate}
\end{definition}

\begin{remark}
Let $(M,g)$ be a complete Riemannian manifold.
\begin{enumerate}
\item  By definition of the Riemannian exponential map $\exp_p\colon T_p M\to M$ at $p\in M$ we have
$$  \exp_p( tv) = \gamma_v(t)   \quad \text{for all}\,\,\, t>0\text{ and }v\in T_pM,  $$
where $\gamma_v$ is the unique geodesic starting at $p$ with $\Dot{\gamma}(0) = v$.
Consequently, $\exp_p$ maps the tangent cut locus $\Cuttil_p(M)$ onto the cut locus $\Cut_p(M)$.
\item We recall the definition of the \textit{global Riemannian exponential map}, see e.g. \cite[p. 128]{lee:2018}, which is given by
$$  \Exp\colon TM\to M\times M,\quad \Exp(v) = (\mathrm{pr}(v),\exp_{\mathrm{pr}(v)}(v)).    $$
Here, $\mathrm{pr}\colon TM\to M$ denotes the bundle projection.
It is clear from the definitions that $\Exp$ maps the total tangent cut locus $\Cuttil(M)$ onto the total cut locus $\Cut(M)$.
\end{enumerate}
\end{remark}

Finally, we want to note how the total cut locus of a Riemannian manifold $(M,g)$ can be used to study the geodesic complexity of $M$. 
As Recio-Mitter argues in \cite[Theorem 3.3]{reciomitter:2021} there is a unique continuous geodesic motion planner on $(M\times M)\setminus \Cut(M)$.
By \cite[Lemma 4.2]{blaszczyk:2018} the latter is an open subset of $M\times M$, from which one derives the estimate
\begin{equation} \label{eq_cut_inequality}
    \GC_{(M,g)}(\Cut(M)) \leq \GC(M) \leq \GC_{(M,g)}(\Cut(M)) + 1  .
\end{equation}
Hence, in order to find bounds on the geodesic complexity of a complete Riemannian manifold $(M,g)$ one can study the subspace geodesic complexity of its total cut locus $\Cut(M)$.

\section{Fibered decompositions of cut loci} \label{sec_fibered_decomp}
 
In this section we introduce the notion of a fibered decomposition of the total cut locus of a Riemannian manifold $M$ and show that such a fibered decomposition of $\Cut(M)$ can be used to derive upper and lower bounds on the geodesic complexity of $M$.
After that we give a condition on the cut locus of a point $p\in M$ of a homogeneous Riemannian manifold which implies that the total cut locus admits a fibered decomposition.

\begin{definition} \label{def_fibered_decomp}
Let $(M,g)$ be a complete Riemannian manifold.
\begin{enumerate}
    \item 
A locally compact decomposition $A_1,\ldots, A_k$ of $\Cut(M)$ is called a \textit{fibered decomposition of} $\Cut(M)$ if the following holds:
for each $i \in \{ 1,\ldots, k\}$ the restricted exponential map
$$ \pi_i = \Exp|_{\widetilde{A}_i} \colon \widetilde{A}_i \to A_i       $$
is a fibration, where $\widetilde{A}_i = \Exp^{-1}(A_i) \cap \Cuttil(M)  $.
\item Similarly, if $p\in M$, then a locally compact decomposition $B_1,\ldots, B_k$ of $\Cut_p(M)$ is called a \textit{fibered decomposition of} $\Cut_p(M)$ if 
$$   \exp_p|_{\widetilde{B}_i} \colon \widetilde{B}_i\to B_i    $$
is a fibration, where $\widetilde{B}_i = \exp_p^{-1}(B_i)\cap \Cuttil_p(M)  $.
\end{enumerate}
\end{definition} 
 
 Here, under a fibration we always understand a Hurewicz fibration in the sense of homotopy theory. Next we will discuss how fibered decompositions of cut loci yield new lower and upper bounds for geodesic complexity.

 \begin{theorem} \label{theorem_fibered_decomp}
 Let $(M,g)$ be a complete Riemannian manifold.
 If the total cut locus $\Cut(M)$ admits a fibered decomposition $A_1,\ldots, A_k$ with fibrations $\pi_i\colon \widetilde{A}_i\to A_i$ for $i\in\{1,\ldots,k\}$ as in Definition \ref{def_fibered_decomp}.(1), then the geodesic complexity of $M$ can be estimated by
 $$ \GC(M) \leq \sum_{i=1}^k \secat(\pi_i)  + 1.      $$
 \end{theorem}
 \begin{proof}
 We begin by showing that continuous local sections of $\pi_i$ induce continuous geodesic motion planners.
 Let $C\subseteq A_i$ be a locally compact subset of $A_i$ and assume that $s\colon C\to \widetilde{A}_i$ is a continuous section of the fibration $\pi_i$.
In particular, we have for $(p,q)\in C$ that
$$ \Exp(s(p,q)) =  (p, q) .       $$
We define $\sigma\colon C\to GM$ by
$$ \sigma ((p,q))(t) =   \mathrm{pr}_2 (\Exp( t s(p,q))) \quad \text{for}\,\,\, t\in [0,1],      $$
where $\mathrm{pr}_2\colon M\times M\to M$ denotes the projection onto the second component.
This is clearly a geodesic motion planner. 
In order to see that map $\sigma$ is also continuous note that the map
$$ \widetilde{\sigma} \colon  C\times I\to M, \quad ((p,q),t)\mapsto \mathrm{pr}_2(\Exp(t s(p,q)))    $$
is continuous since it is a composition of continuous maps.
By a general property of the compact-open topology, the continuity of $\widetilde{\sigma}$ implies the continuity of the induced map $\sigma \colon C\to GM$, see e.g. \cite[Theorem VII.2.4]{bredon:2013}.

For each $i\in\{ 1,\ldots,k\}$ we put $m_i := \secat(\pi_i)$. Then, see e.g. \cite[Lemma 4.1]{mescher:2021}, for each $i$ there is a locally compact decomposition $C_{i,1},\ldots,C_{i,m_i}$ of $A_i$ for which there is a continuous section of $\pi_i$ on each $C_{i,j}$, $j \in \{1,\dots,m_i\}$.
Since the sets $A_1,\ldots, A_k$ form a decomposition of $\Cut(M)$, we see that the sets
$$\{C_{i,j}\,|\,  i\in\{1,\ldots,k\}, j\in\{1,\ldots,m_i\} \}$$
are a decomposition of $\Cut(M)$ with each $C_{i,j}$ locally compact. 
By the first part of the proof we see that each $C_{i,j}$ admits a continuous geodesic motion planner.
This shows that
$$    \GC_{(M,g)}(\Cut(M)) \leq \sum_{i=1}^k m_i = \sum_{i=1}^k \secat(\pi_i) .  $$
 Combining this inequality with the inequality \eqref{eq_cut_inequality} completes the proof.
 \end{proof}
 

In the subsequent sections we will see examples of upper bounds on geodesic complexity by virtue of Theorem \ref{theorem_fibered_decomp}.
The next result however shows how a fibered decomposition of the total cut locus $\Cut(M)$ gives rise to a lower bound on $\GC_{(M,g)}(\Cut(M))$.
Before we state the result, we recall the definition of the \textit{velocity map}, see \cite[Definition 3.1]{mescher:2021}, i.e. the map  given by
$$   v:GM \to TM, \qquad  v(\gamma) = \Dot{\gamma}(0).$$
The velocity map is continuous by \cite[Proposition 3.2]{mescher:2021}.
Furthermore, we recall that the sectional category of a fibration $p\colon E\to B$ is defined by considering open covers $U_1,\ldots,U_k$ of $B$ such that each $U_i$, $i\in\{1,\ldots,k\}$ admits a continuous local section of $p$.
The geodesic complexity of a complete Riemannian manifold $M$ however is defined via locally compact decompositions of $M\times M$.
In order to compare these two concepts in the following theorem, we employ the notion of \textit{generalized sectional category} as introduced by J. M. Garc\'{\i}a Calcines in \cite[Definition 2.1]{garcia:2019}.
\begin{definition}
Let $p\colon E\to B$ be a fibration. The \textit{generalized sectional category} $\secat_g(p)$ is defined as the smallest integer $k$ for which there exists a cover $A_1,\ldots, A_k$ of $B$ such that each $A_i$, $i\in\{1,\ldots,k\}$, admits a continuous local section of $p$.
\end{definition}
Note that the sets $A_i$ in the above definition can be arbitrary subsets of $B$.
García-Calcines shows in \cite[Theorem 2.7]{garcia:2019} that if $p\colon E\to B$ is a fibration and if $E$ and $B$ are absolute neighborhood retracts, one has
$$   \secat_g(p) = \secat(p) .    $$

\begin{theorem} \label{theorem_lower_bound}
Let $(M,g)$ be a complete Riemannian manifold.
Assume that the total cut locus $\Cut(M)$ admits a fibered decomposition $A_1,\ldots, A_l$ with fibrations $\pi_i \colon \widetilde{A}_i\to A_i$ for $i\in\{1,\ldots,l\}$.
Furthermore, assume that all $\widetilde{A}_i$ and $A_i$ are absolute neighborhood retracts.
Then 
$$  \GC_{(M,g)}(\Cut(M)) \geq \max\{ \secat(\pi_i)\,|\, i\in \{1,\ldots, l\}\}  .     $$
\end{theorem}
\begin{proof}
Let $m\in \NN$ be the maximum of $\{\secat(\pi_i)\,|\,i\in\{1,\ldots,l\}\}$ and choose $i_0\in\{1,\ldots,l\}$ such that $\secat(\pi_{i_0}) = m$.
Assume that the assertion of the theorem is false.
Then there are a locally compact decomposition $B_1,\ldots, B_k$ of $\Cut(M)$ with $k< m$ and continuous geodesic motion planners $s_j \colon B_j\to GM$ for $j \in \{1,\dots,k\}$.
For $i\in \{1,\ldots,k\}$ set $C_i = B_i\cap A_{i_0}$.
It is possible that there are $i\in \{1,\ldots,k\}$ with $C_i = \emptyset$. 
By reordering the $B_i$ we can arrange that $C_1,\ldots,C_r\neq \emptyset$ and $C_{r+1},\ldots,C_k = \emptyset$ for some $1\leq r\leq k$.
The sets $C_1,\ldots, C_r$ form a cover of $A_{i_0}$.
For $j\in\{1,\ldots,r\}$ we define a map 
$$\sigma_j \colon C_j \to \widetilde{A}_{i_0}, \qquad  \sigma_j = v\circ s_j|_{C_j} ,    $$ 
where $v$ denotes the velocity map. It is clear that $\sigma_j$ is continuous.
We claim that it is a section of $\widetilde{A}_{i_0}$.
For any $(p,q)\in C_j$ the path $s_j(p,q)$ is a minimal geodesic. Thus, there is $w\in \Cuttil_p(M)$ with 
$$   s_j(p,q)(t) =  \exp_p(tw) .      $$
By definition of the velocity map, we obtain
$$  \sigma_j(p,q) = (v\circ s_j)(p,q) = w     $$
and by definition of $\widetilde{A}_{i_0}$ it is clear that $w\in \widetilde{A}_{i_0}$.
Consequently,
$$   (\pi_{i_0}\circ \sigma_j)(p,q) = (\Exp|_{\widetilde{A}_{i_0}}\circ \sigma_j)(p,q) = (p,\exp_p(w)) = (p,q),      $$
which shows that $\sigma_j$ is a continuous section of $\pi_{i_0}$.
Hence, we obtain
$$  \secat_g(\pi_{i_0}) \leq r \leq k < m   . $$
Since $\pi_{i_0}\colon \widetilde{A}_{i_0}\to A_{i_0}$ is a fibration with $\widetilde{A}_{i_0}$ and $A_{i_0}$ being absolute neighborhood retracts, we derive from \cite[Theorem 2.7]{garcia:2019} that
$$  \secat(\pi_{i_0}) = \secat_g(\pi_{i_0})  < m ,  $$
which is a contradiction.
This completes the proof.
\end{proof}
\begin{cor} \label{cor_single_fibration}
Let $(M,g)$ be a complete Riemannian manifold. Assume that $$\pi = \Exp|_{\Cuttil(M)}\colon \Cuttil(M)\to \Cut(M)$$ is a fibration and assume that $\Cuttil(M)$ and $\Cut(M)$ are absolute neighborhood retracts. 
Then
$$   \GC_{(M,g)}(\Cut(M)) = \secat(  \pi)   \quad \text{and} \quad \secat(\pi)\leq \GC(M,g)\leq \secat(\pi)+ 1 . $$
\end{cor}
\begin{proof}
It is clear by Theorem \ref{theorem_lower_bound} that 
$$  \GC_{(M,g)}(\Cut(M))\geq \secat(\pi) .    $$
The reverse inequality follows from the proof of Theorem \ref{theorem_fibered_decomp}.
The second asserted inequality follows from equation \eqref{eq_cut_inequality}.
\end{proof}

In Section \ref{sec_symm} we will show that the symmetric metrics on complex and quaternionic projective spaces are examples for which the conditions of Corollary \ref{cor_single_fibration} are satisfied.


In the following we will derive a tangible criterion in order to find fibered decompositions of the total cut locus.
In the setting of homogeneous Riemannian manifolds we want to use a fibered decomposition of the cut locus of a point to obtain a fibered decomposition of the total cut locus, whose fibrations will in fact be fiber bundles.
We will see applications of this idea in Sections \ref{sec_symm} and \ref{sec_lens}.

Note that if a compact group of isometries acts transitively on a Riemannian manifold, then the manifold is necessarily complete. 
If $K$ is a group of isometries of a Riemannian manifold which fixes a point $p\in M$, then $k \cdot \Cut_p(M) = \Cut_p(M)$ for all $k\in K$.	

\begin{definition}
Let $(M,g)$ be a Riemannian manifold and assume that $G$ is a group of isometries acting transitively on $M$.
Let $p\in M$ be a point and let $K\subseteq G$ be its isotropy group.
Let $ B_1,\ldots,B_m$ be a locally compact decomposition of $\Cut_p(M)$.
We say that the decomposition is \textit{isotropy-invariant} if $k \cdot B_i = B_i$ for all $i = 1,\ldots, m$ and all $k\in K$.
\end{definition}

In the following let $(M,g)$ be a Riemannian manifold and let $G$ be a group of isometries of $M$ acting transitively on $M$.
We denote the group action by $\Phi\colon G\times M\to M$.
We shall use the shorthand notation $\Phi_g  = \Phi(g,\cdot)\colon M\to M$ as well as $\Phi(g,p) = g \cdot p$ for $g\in G,p\in M$.

Our aim is to use the homogeneity of $M$ to construct a fibered decomposition of the total cut locus $\Cut(M)$ out of a fibered decomposition of the cut locus of one single point in $M$. \medskip 

\emph{In the following, we fix a point $p\in M$ and let $B_1,\ldots, B_k$ be a decomposition of $\Cut_p(M)$ which is both isotropy-invariant and a fibered decomposition such that the associated fibrations $\widetilde{B}_i\to B_i$ are fiber bundles for $i \in \{1,\ldots,k\}$.}\medskip 

Let $K$ be the isotropy group of $p$ and let $\mathrm{pr}: G\to M \cong G/K$ denote the canonical projection.
For $i\in\{1,\ldots,k\}$ set
$$  A_i = \{ (q,r)\in \Cut(M)\,|\, r \in \Phi_g(B_i) \,\, \text{for some} \,\,g\in G\,\,\text{with}\,\,\mathrm{pr}(g) = q \}     $$
and
$$   \widetilde{A}_i = \{ (q,v)\in \Cuttil(M)\,|\, v\in (D \Phi_g)_p( \widetilde{B}_i)  \,\,\text{for some} \,\,g\in G\,\,\text{with}\,\,\mathrm{pr}(g) = q \} .    $$
We further consider the maps 
$$\pi_i: A_i \to M, \quad \pi_i(q,r)=q, \qquad \widetilde{\pi}_i: \widetilde{A}_i \to M, \quad \widetilde{\pi}_i(q,v)=q, \qquad i \in\{1,\ldots,k\}.$$

\begin{lemma} \label{lemma_fiber_bundle}
In the present setting the following holds for each $i \in \{1,\dots,k\}$:
\begin{enumerate}
    \item  $\pi_i: A_i\to M$ is a fiber bundle with typical fiber $B_i$.
\item $\widetilde{\pi}_i:\widetilde{A}_i \to M$ is a fiber bundle with typical fiber $\widetilde{B}_i$.
\end{enumerate}
\end{lemma}
Note that by fiber bundle, we mean a fiber bundle in the continuous category.
We do not assume that the sets $B_i$ carry any differentiable structure.%
\begin{proof}
We want to show that both $A_i$ and $\widetilde{A}_i$ are locally trivial.
Fix an $i\in\{1,\ldots,k\}$ and let
$$ \pi_i\colon A_i\to M, \quad \pi_i(q,r) = q , $$
be the projection on the first factor.
Let $U\subseteq M$ be an open set on which there exists a continuous section $s\colon U\to G$ of $\mathrm{pr}$.
Define $\varphi_i\colon A_i|_U\to U\times B_i$ by\
$$  \varphi_i(q,r) = (q, s(q)^{-1} \cdot r) \quad \text{for}\quad (q,r)\in A_i|_U .     $$
This is a well-defined map since if $(q,r)\in A_i$, then there is a $b\in B_i$ such that $r = g\cdot b$ for some $g\in G$ with $g\cdot p = q$.
Therefore, by the isotropy invariance of the decomposition $B_1,\ldots, B_k$,
$$   s(q)^{-1} \cdot r = (s(q)^{-1} g)\cdot  b \in B_i   $$
since $s(q)^{-1} g\in K$. Evidently, $\varphi_i$ is a homeomorphism.
For each point $(q,r)\in A_i$ there is such an open neighborhood $U$ of $q$ admitting a continuous section $s\colon U\to G$ of $\mathrm{pr}$. Thus, the above construction shows that $A_i\to M$ is a continuous fiber bundle.
The proof for $\widetilde{A}_i$ is analogous.
One defines local trivializations of the form $\psi_i\colon A_i|_U\to U\times \widetilde{B}_i$, where $U$ is an open subset of $M$ admitting a continuous section $s\colon U\to G$ of $\mathrm{pr}$, by 
$$  \psi_i(q,v) = ( q, (D\Phi_{s(q)^{-1}})_q v) \quad \text{for}\,\,\, (q,v)\in\widetilde{A}_i  .   $$

As for $\varphi_i$ one shows that $\psi_i$ is well-defined and a homeomorphism.
\end{proof}

\begin{theorem} \label{theorem_isotropy_fibered}
Let $(M,g)$ be a Riemannian manifold and $G$ be a group of isometries of $M$ acting transitively on $M$.
Fix a point $p\in M$.
Let $B_1,\ldots, B_k$ be a decomposition of $\Cut_p(M)$ which is both isotropy-invariant and a fibered decomposition such that the associated fibrations $\widetilde{B}_i\to B_i$ are fiber bundles. 
For $i=1,\ldots,k$ let $C_i$ be the typical fiber of the bundle $\widetilde{B}_i\to B_i$.
Define the sets $A_i\subseteq \Cut(M)$ as above.
Then the decomposition of $\Cut(M)$ into $A_1,\ldots, A_k$ is a fibered decomposition.
More precisely, the restriction $\Exp|_{\widetilde{A}_i}\colon \widetilde{A}_i\to A_i$ is a fiber bundle with typical fiber $C_i$.
\end{theorem}
\begin{proof}
Fix $i\in\{1,\ldots,k\}$ and let $p \in M$. As discussed in the proof of Lemma \ref{lemma_fiber_bundle}, we can find an open neighborhood $U\subseteq M$ of $p$ and local trivializations $\varphi_i\colon A_i |_U\to U\times B_i$ and $\psi_i\colon \widetilde{A}_i|_U\to U\times \widetilde{B}_i$. If $\varphi_i$ and $\psi_i$ are given as in that proof, then the inverse of $\varphi_i$ is explicitly  given by $$ \varphi_i^{-1}\colon U\times B_i\to A_i|_U, \quad \varphi_i^{-1}(q,b) = (q,s(q)\cdot b),$$
where $s\colon U\to G$ is a local section of $\mathrm{pr}\colon G\to M$. We claim that the diagram
$$
\begin{tikzcd}
\widetilde{A}_i|_U \arrow[]{r}{\Exp|_{\widetilde{A}_i|_U}} \arrow[]{d}{\psi_i} 
& [3em]
A_i|_U \arrow[swap]{d}{\varphi_i} 
\\
U\times \widetilde{B}_i \arrow[swap]{r}{(\id_U,\exp_p)} 
&
U\times B_i 
\end{tikzcd}
$$
commutes. To see this, let $(q,v)\in\widetilde{A}_i|_U$.
Then
$$ \psi_i(q,v) = (q, (D \Phi_{s(q)^{-1}})_q v) = (q, (D\Phi_{s(q)})_p^{-1} v) \in U\times \widetilde{B}_i .     $$
By naturality of the exponential map, see \cite[Proposition 5.20]{lee:2018}, it thus holds that
\begin{eqnarray*}
 (\varphi_i^{-1}\circ(\id_U,\exp_p) \circ \psi_i)(q,v) &=& \varphi_i^{-1}(q, \exp_p(   (D\Phi_{s(q)})_p^{-1} v)) \\
  &=& (q, s(q)  s(q)^{-1} \cdot \exp_q(v))  \\ &=& (q,\exp_q(v))  \\ &=& \Exp(q,v).
\end{eqnarray*}
By assumption the restriction $\exp_p|_{\widetilde{B}_i}\colon \widetilde{B}_i\to B_i$ is a fiber bundle. 
Hence, by choosing an open subset $V\subseteq B_i$ such that $\widetilde{B}_i|_{V}$ is trivial and considering $\varphi^{-1}_i(U\times V)$ we obtain an open set in $A_i$ over which the map $\Exp|_{\widetilde{A}_i}\colon \widetilde{A}_i\to A_i$ is trivial.
Since $A_i$ is covered by such trivializations, this proves the claim.
\end{proof}

\section{The total cut loci of symmetric spaces} \label{sec_symm}

In this section we turn to the study of cut loci in irreducible compact simply connected symmetric spaces and show that the total cut locus of these spaces always admits a fibered decomposition.
Furthermore, we derive a new upper bound for the geodesic complexity of symmetric spaces.
Note that this section is related to \cite[Section 8]{mescher:2021} where the authors proved an upper bound for irreducible compact simply connected symmetric spaces in terms of the sectional category of the isometry bundle $\Isom(M)\to M$ over a symmetric space $M$ and certain subspace geodesic complexities.
The upper bound in the current section is derived independently of this previous result.

We briefly recall the most important notions related to root systems of symmetric spaces.
Let $M=G/K$ be a symmetric space with $(G,K)$ being a Riemannian symmetric pair.
Denote the canonical projection by $\pi\colon G\to G/K\cong M$.
There is a decomposition $\gg = \kk \oplus\mm$ of the Lie algebra $\gg$ of $G$ such that $\mm \cong T_{\pi(e)}M$ is a linear isometry.
We set $o=\pi(e)$, where $e$ is the unit element of $G$.
Consider the complexification $\gg_{\CC}$ of $\gg$ and choose a Cartan subalgebra $\hh\subseteq \gg_{\CC}$.
A \textit{root} of $\gg_{\CC}$ is an element $\alpha\in\hh^*$ of the dual space of $\hh$ for which there exists an $X\in\gg_{\CC}\setminus \{0\}$ with
$$  [H,X] = \alpha(H) X \quad \text{for all}\,\,\,H\in \hh.    $$
If $\aa$ is a maximal abelian subalgebra of $\mm$, then consider the restriction $\alpha|_{\aa}$ of a root of $\gg_{\CC}$.
If this restriction is non-zero, we call it a \emph{root} of the symmetric pair $(G,K)$.
We choose and fix a set of \emph{simple roots} of the symmetric pair $(G,K)$ and denote it by $\pi(G,K)$. We further let $\delta$ denote its \emph{highest root}. See \cite[Section X.3]{helgason:78} or \cite[Section V.4]{broecker:85} for details on these notions.
Due to the compactness of $G$ we can choose an $\Ad_G$-invariant inner product $\left<\cdot,\cdot\right>$ on $\gg$ and identify the roots with vectors in $\aa$ via this inner product.
Then a \emph{Weyl chamber} of $\pi(G,K)$ can be defined as
$$W:= \{X \in \aa \ | \left<\gamma,X \right> >0 \ \ \forall \gamma \in \pi(G,K)\}. $$
Note that one can define the other Weyl chambers by choosing other systems of simple roots.
The \emph{Weyl group} $W(G,K)$ of the symmetric pair $(G,K)$ is generated by the reflections $s_{\alpha}$ on the hyperplanes 
$$  \{ H\in \aa \,|\, \alpha(H) = 0\} .   $$
It is a finite group and acts simply transitively on the set of Weyl chambers of $(G,K)$.\medskip

T. Sakai has studied the cut loci of compact simply connected symmetric spaces in \cite{Sakai1}, see also \cite{Sakai3} and \cite{Sakai2}.
We summarize the main results.
If there are two or more simple roots of $(G,K)$, put
$$\D := \{\Delta \subset \pi(G,K) \ | \ \Delta \neq \emptyset, \ \delta \notin \Delta\}  . $$
In case there is only one simple root $\gamma$, this is then also the highest root and we set
$$  \D : = \{ \{\gamma \} \}  .  $$
If there are two or more simple roots, we set 
$$S_\Delta := \left\{X \in \overline{W} \ | \ \left<\gamma,X\right> >0 \ \  \forall \gamma \in \Delta, \ \left<\gamma,X\right>=0 \ \ \forall \gamma \in \pi(G,K)\setminus \Delta, \ 2 \left<\delta,X\right>=1  \right\}$$
for each $\Delta \in \D$.
In case there is a single simple root $\gamma$, we define
$$  S_{\{\gamma\}} := \{ X\in \aa\, | \, 2\langle \gamma, X \rangle = 1\}   .  $$
As usual, we denote by $\exp:\gg\to G$ the exponential map of $G$ and define 
 $$ \overline{\exp}:\mm \to M, \qquad \overline{\exp} := \pi \circ \exp|_{\mm}.$$
This in fact agrees with the Riemannian exponential at the point $o$ under the canonical identification $\mm \cong T_o M$ and is often denoted by $\Exp$.
In order not to confuse it with the global Riemannian exponential map used in Section \ref{sec_fibered_decomp}, we denote it by $\overline{\exp}$.
For $\Delta \in \D$ set
$$\widetilde{\Phi}_\Delta: K \times S_\Delta \to M, \qquad \widetilde{\Phi}_\Delta(k,X) = \overline{\exp}(\Ad_k(X)) $$
and 
$$ \widetilde{\Psi}_{\Delta} \colon K\times S_{\Delta}\to \mm, \qquad \widetilde{\Psi}_{\Delta}(k,X) = \Ad_k(X) .    $$
Furthermore, we define  $$Z_\Delta := \{ k \in K \ | \ \overline{\exp}(\Ad_k(X))=\overline{\exp}(X) \ \forall X \in S_\Delta\}$$
and 
$$ K_{\Delta} = \{ k\in K\,|\, \Ad_k(X) = X  \  \forall X\in S_{\Delta}    \} .$$
Evidently, $Z_\Delta$ and $K_{\Delta}$ are closed subgroups of $K$ with $K_{\Delta}\subseteq Z_{\Delta}$. 
Sakai shows in \cite[Proposition 4.10]{Sakai1}, that if $\Delta\in \D$, then the map $\widetilde{\Phi}_\Delta$ induces a differentiable embedding
$$\Phi_{\Delta}: K/Z_\Delta \times S_\Delta \to M.$$
Define $C_\Delta := \im \Phi_\Delta$ for each $\Delta \in \D$. 
The cut locus of the point $o=\pi(e)\in M$ is then given by 
$$\Cut_o(M) = \bigcup_{\Delta \in \D} C_\Delta $$
see \cite[Theorem 5.3]{Sakai1}. Moreover, the set $\{C_{\Delta}\}_{\Delta\in\D}$ forms a locally compact decomposition of $\Cut_o(M)$.
\begin{lemma} \label{lemma_psi_delta}
The map $\widetilde{\Psi}_{\Delta}$ induces a continuous embedding $\Psi_{\Delta}\colon K/K_{\Delta}\times S_{\Delta}\to \mm$ and for $\widetilde{C}_{\Delta} := \mathrm{im}(\Psi_{\Delta})$ we have that
$$  \widetilde{C}_{\Delta} = \overline{\exp}^{-1}(C_{\Delta}) \cap \Cuttil_o(M)    .     $$
\end{lemma}
\begin{proof}
By definition of $K_{\Delta}$ it is clear that $\widetilde{\Psi}_{\Delta}$ induces a continuous map $\Psi_\Delta: K/K_{\Delta}\times S_{\Delta}\to \mm$. To prove that $\Psi_{\Delta}$ is an embedding, we closely follow the proof of \cite[Proposition 4.10]{Sakai1}. For the injectivity of $\Psi_{\Delta}$, let $k,k'\in K$ and $X,X'\in S_{\Delta}$ such that $\Ad_{k'}X' = \Ad_k X$. 
We need to show that $[k']= [k]$ in $K/K_{\Delta}$ and that $X=X'$.
Clearly, it holds that $$\Ad_{k^{-1}k'}X' = X.$$
Therefore, by \cite[Proposition VI.2.2]{helgason:78} we know that there is an element $s$ of the Weyl group $W(G,K)$ of the Riemannian pair $(G,K)$ such that $sX'= X$.
But since $X$ and $X'$ are in the closure of the same Weyl chamber, they have to be equal, see \cite[p. 131]{Sakai1}.
This also shows that $k^{-1}k'\in K_{\Delta}$, so $[k'] = [k]$ in $K/K_{\Delta}$.

In order to show that $\Psi_{\Delta}$ is an embedding, let $(k_n)_{n\in\NN}$ be a sequence in $K$ and $(X_n)_{n\in\NN}$ be a sequence in $S_{\Delta}$ such that $\Ad_{k_n}(X_n) \to \Ad_k X$ for $n\to \infty$, where $k\in K$ and $X\in S_{\Delta}$.
We want to show that $[k_n]\to [k]$ in $K/K_{\Delta}$ and $X_n\to X$ for $n\to\infty$.
Assume that this does not hold.
Then by compactness of $K$ there are $k_0\in K$ and $Y\in \overline{W}$ and there are subsequences $(k_{n_i})_{i\in\NN}$ and $(X_{n_i})_{i\in\NN}$ with $k_{n_i}\to k_0$ and $X_{n_i}\to Y$ for $i\to\infty$ with $([k_0],Y)\neq ([k],X)$.
By continuity of $\Ad$ we have $\Ad_{k_0}Y = \Ad_k X$ so as argued above for the injectivity, we obtain $X= Y$ and $[k_0] = [k]$ in $K/K_{\Delta}$ which gives a contradiction.
This shows the sequential continuity of $\Phi_\Delta^{-1}$, thereby yielding that $\Psi_{\Delta}$ is an embedding.

Finally, by \cite[p.133]{Sakai1} we have that $\widetilde{C}_{\Delta} = \mathrm{im}(\Psi_{\Delta})\subseteq \Cuttil_o(M)$. 
Moreover, it is clear by construction that $ \overline{\exp}(\widetilde{C}_{\Delta}) = C_{\Delta}     $.
In order to show that 
$$  (\overline{\exp}|_{\Cuttil_o(M)})^{-1}(C_{\Delta})\subseteq \widetilde{C}_{\Delta}   $$
let $X\in \Cuttil_o(M)$ such that $\overline{\exp}(X)\in C_{\Delta}$.
Then there is $k\in K$ with $$k \cdot \overline{\exp}(X) = \overline{\exp}(\Ad_k(X)) = \overline{\exp}(Y)$$ for some $Y\in S_{\Delta}$.
We set $q = \overline{\exp}(Y)$ and $\widetilde{X} = \Ad_k X$.
Clearly, $\widetilde{X}\in\Cuttil_o(M)$ and since $\widetilde{X}\in \overline{\exp}^{-1}(q)$ we have by \cite[Lemma 4.7]{Sakai1} that there is an $h\in Z_{\Delta}$ with $\widetilde{X} = \Ad(h)(Y)$.
But this implies that $X = \Ad(k^{-1} h )(Y)$ which shows that $X\in\widetilde{C}_{\Delta}$.
\end{proof}
It is clear by construction that the decomposition $\{C_{\Delta}\}_{\Delta\in\D}$ of $\Cut_o(M)$ is isotropy-invariant.
The next theorem shows that it is a fibered decomposition of $\Cut_o(M)$.

\begin{theorem} \label{theorem_symm_fibered}
Let $M = G/K$ be an irreducible compact simply connected symmetric space with $(G,K)$ being a Riemannian symmetric pair and let $p \in M$. Then the cut locus of $p$ admits a decomposition which is both isotropy-invariant and a fibered decomposition with the associated fibrations being fiber bundles.
\end{theorem}
\begin{proof}
As we have already argued, the decomposition of $\Cut_o(M)$ into the $C_{\Delta}$, $\Delta \in \D$, is a decomposition into locally compact subsets and is isotropy-invariant.
Hence, it remains to show that it is a fibered decomposition.
Let $\Delta\in \D$ and consider the map
$$\chi\colon K/K_{\Delta}\times S_{\Delta}\to K/Z_{\Delta}\times S_{\Delta}, \qquad \chi( k K_{\Delta}, X) = ( k Z_{\Delta},X) .    $$
We derive from Lemma \ref{lemma_psi_delta} that the diagram
$$ 
\begin{tikzcd}
K/K_{\Delta} \times S_{\Delta} \arrow[]{r}{\chi} \arrow[]{d}{\Psi_{\Delta}} &
K/Z_{\Delta}\times S_{\Delta} \arrow[]{d}{\Phi_{\Delta}} 
\\
\widetilde{C}_{\Delta} \arrow[]{r}{\overline{\exp}|_{\widetilde{C}_\Delta}} & C_{\Delta} 
\end{tikzcd} 
$$
commutes where the vertical arrows are homeomorphisms. It is well-known, see e.g. \cite[Theorem I.7.4]{steenrod:2016}, that the canonical map $K/K_{\Delta}\to K/Z_{\Delta}$ is a fiber bundle with typical fiber $Z_{\Delta}/K_{\Delta}$.
Consequently, the above commutative diagram shows that $\overline{\exp}|_{\widetilde{C}_{\Delta}}\colon \widetilde{C}_{\Delta}\to C_{\Delta}$ is a fiber bundle with typical fiber $Z_{\Delta}/K_{\Delta}$.
Since this holds for all $\Delta\in\D$ we have shown that the decomposition $\{C_{\Delta}\}_{\Delta\in\D}$ is a fibered decomposition with the associated fibrations being fiber bundles.
\end{proof}
Combining Theorems \ref{theorem_isotropy_fibered} and \ref{theorem_symm_fibered} we obtain the following.
\begin{cor} \label{cor_fibered_symm}
Let $M$ be an irreducible compact simply connected symmetric space.
Then the total cut locus $\Cut(M)$ admits a fibered decomposition and the associated fibrations are fiber bundles.
\end{cor}
For $\Delta \in \D$ let $A_{\Delta}\subseteq \Cut(M)$ and $\widetilde{A}_{\Delta}\subseteq \Cuttil(M)$ be the subsets of the total cut locus and the total tangent cut locus, resp., induced by the $C_{\Delta}$ as described in Section \ref{sec_fibered_decomp}.
The set $\pi(G,K)$ consists of precisely $r=\mathrm{rank} \, M$ elements. 
For each $i \in \{1,2,\dots,r\}$ we set 
$$\D_i := \{ \Delta \in \D \ | \ \#\Delta=i \} \qquad \text{and}\qquad A_i := \bigcup_{\Delta \in \D_i} A_\Delta.$$ 
Note that by \cite[Lemma 5.2]{Sakai1}, we have for all $i\in\{1,\ldots,r\}$ that
$$   \overline{C}_{\Delta} \cap C_{\Delta'}  = \emptyset \qquad \text{for}\,\,\,\Delta,\Delta'\in \D_i,  \,\, \Delta\neq \Delta' .     $$
It is easy to see that the same relation then holds for the $A_{\Delta}$, i.e.
\begin{equation}
\label{eq_Adelta}
   \overline{A}_{\Delta} \cap A_{\Delta'}  = \emptyset \qquad \text{for}\,\,\,\Delta,\Delta'\in \D_i  , \,\, \Delta\neq \Delta' .     
  \end{equation}
Therefore, if we have a locally compact decomposition of all $A_{\Delta}$, $\Delta\in\D_i$, then we can combine geodesic motion planners in the following way.
\begin{theorem} \label{theorem_upp_bound_symm_spaces}
Let $M$ be an irreducible compact simply connected symmetric space of rank $r$.
Then the geodesic complexity of $M$ can be estimated by
$$   \GC(M) \leq \sum_{i=1}^r \max\{ \secat(\Exp|_{\widetilde{A}_{\Delta}}\colon \widetilde{A}_{\Delta} \to A_{\Delta} )\,|\, \Delta \in\D_i\} + 1 .     $$
\end{theorem}
\begin{proof}
Let $i\in\{1,\ldots,r\}$ and assume that for each $\Delta\in\D_i$ we have a locally compact decomposition $B_{\Delta,1},\ldots,B_{\Delta,k_{\Delta}}$ of $A_\Delta$ such that for each $j\in\{1,\ldots,k_{\Delta}\}$ there is a continuous geodesic motion planner $s_{\Delta,j}\colon B_{\Delta,j}\to GM$.
Let $m_i = \max\{ k_{\Delta}\,|\, \Delta\in\D_i\} $ and set $B_{\Delta,j} = \emptyset$ for $k_{\Delta}< j\leq m_i$.
For $l= 1,\ldots,m_i$ put
$$   B_l = \bigcup_{\Delta\in\D_i} B_{\Delta,l}     $$
and define a geodesic motion planner $s_l\colon B_l\to GM$ by
$$  s_l (q,r) = s_{\Delta,l}(q,r) \qquad \text{if}\,\,\, (q,r)\in B_{\Delta,l} .     $$
It follows from \eqref{eq_Adelta} that this defines a continuous geodesic motion planner on $B_l$. Since the sets  $B_1,\dots,B_{m_i}$ form a decomposition of $A_i$, this shows that $\GC_M(A_i)\leq m_i$.
Arguing as in the proof of Theorem \ref{theorem_fibered_decomp}, one further shows that 
$$k_{\Delta} \leq \secat( \Exp|_{\widetilde{A}_{\Delta}}\colon \widetilde{A}_{\Delta} \to A_{\Delta} ) \qquad \forall \Delta \in \D_i,$$ 
which in turn yields $m_i \leq \max\{\secat(\Exp|_{\widetilde{A}_{\Delta}}) \ | \ \Delta \in \D_i\}$ for each $i \in \{1,2,\dots,r\}$. Eventually, we derive that
$$\GC_M(\Cut(M)) \leq \sum_{i=1}^r \GC_M(A_i) \leq \sum_{i=1}^r m_i \leq \sum_{i=1}^r \max\{\secat(\Exp|_{\widetilde{A}_{\Delta}}) \ | \ \Delta \in \D_i\}.$$
\end{proof}
Throughout the following, we shall always write $\cong$ to indicate that two manifolds are diffeomorphic. We further let $\mathbb{S}^n$ denote the $n$-sphere with its standard differentiable structure for each $n \in \NN$.

\begin{example}
Consider the complex Grassmannian $\mathrm{Gr}_2(\CC^4)$ which is an irreducible compact symmetric space of rank $2$.
As shown in \cite[p.143]{Sakai1} and \cite[Example 8.5]{mescher:2021}, the cut locus $\Cut_o(M)$ can be decomposed into 
$$  C_{\Delta_1}\cong \mathbb{S}^2\times \mathbb{S}^2, \qquad C_{\Delta_2}\cong \{*\}     $$
and a six-dimensional manifold $C_{\Delta_0}$.
As discussed in \cite[Example 8.5]{mescher:2021}, these three spaces are simply connected.
Note that $\D_1 = \{\Delta_1,\Delta_2\}$.
The decomposition of the cut locus of $o$ induces a decomposition of $\Cut(M)$ as in Section \ref{sec_fibered_decomp} and we shall call the induced sets $A_{\Delta_0}, A_{\Delta_1}$ and $A_{\Delta_2}$.
In order to apply Theorem \ref{theorem_upp_bound_symm_spaces}, we need to find upper bounds for 
$$   \secat(\Exp|_{\widetilde{A}_{\Delta_i}}\colon \widetilde{A}_{\Delta_i}\to A_{\Delta_i}) \quad \text{for}\,\,\, i = 0,1,2.    $$
Fix $i\in\{0,1,2\}$.
By \cite[Theorem 18]{schwarz:1966}, we have $\secat(\pi\colon E\to B)\leq \cat(B)$ for any fibration $\pi$ where $\cat(B)$ is the Lusternik-Schnirelmann category of $B$.
Consequently, we obtain
$$ \secat(\Exp|_{\widetilde{A}_{\Delta_i}}\colon \widetilde{A}_{\Delta_i}\to A_{\Delta_i})  \leq \cat(A_{\Delta_i}) .  $$
Note that $\mathrm{Gr}_2(\CC^4)$ and $C_{\Delta_i}$ are simply connected, therefore $A_{\Delta_i}$ is simply connected since it is a fiber bundle over $\mathrm{Gr}_2(\CC^4)$ with typical fiber $C_{\Delta_i}$ by Lemma \ref{lemma_fiber_bundle}.
Therefore, we get the estimate
$$  \cat(A_{\Delta_i}) \leq \frac{\mathrm{dim}(A_{\Delta_i})}{2} + 1 = \frac{\mathrm{dim}(\mathrm{Gr}_2(\CC^4)) + \mathrm{dim}(C_{\Delta_i})}{2} + 1 = \frac{\dim C_{\Delta_i}}{2}+5    $$
by \cite[Theorem 1.50]{cornea:2003}.
Explicitly, we obtain
$$   \cat(A_{\Delta_0})\leq 8, \quad \cat(A_{\Delta_1})\leq 7 \quad \text{and}\quad \cat(A_{\Delta_2})\leq  5  . $$
Consequently, by Theorem \ref{theorem_upp_bound_symm_spaces}, we see that
$$  \GC(M) \leq \cat(A_{\Delta_0}) + \max\{ \cat(A_{\Delta_1}),\cat(A_{\Delta_2})\} + 1 = 16 .    $$
Note that this improves the upper bound in \cite[Example 8.5]{mescher:2021}.
\end{example}

\begin{theorem} \label{theorem_projective}
Let $M =  \CC P^n$ or $\mathbb{H}P^n$ equipped with the standard or Fubini-Study metric, where $n \in \NN$. 
Then its geodesic complexity satisfies
$$  \GC( M) = 2n + 1 .      $$
In particular, one has
$$   \GC(M) = \TC(M) .    $$
\end{theorem}
\begin{proof}
Since $\CC P^n$ and $\mathbb{H}P^n$ are simply connected symmetric spaces of rank one, we know by \cite[Theorem 5.3]{Sakai1} and Corollary \ref{cor_fibered_symm} that the restriction
$$  \Exp|_{\Cuttil(M)}\colon \Cuttil(M)\to \Cut(M)     $$
is a fibration.
Moreover for $n\geq 2$, the cut locus of a point satisfies
$$  \Cut_p(\CC P^n) \cong \CC P^{n-1}  \quad \text{and}\quad \Cut_q(\mathbb{H}P^n) \cong \mathbb{H}P^{n-1} ,  $$
where $p\in \CC P^n$ and $q\in \mathbb{H}P^n$, see \cite[Proposition 3.35]{besse:1978}.
By Lemma \ref{lemma_fiber_bundle} we see that $\Cut(\CC P^n)$ is a fiber bundle over $\CC P^n$ with typical fiber $\CC P^{n-1}$.
Since $\CC P^n$ is simply connected for each $n\geq 1$, it follows that $\Cut(\CC P^n)$ is simply connected as well for all $n\geq 2$.
By \cite[Theorem 18]{schwarz:1966} and \cite[Theorem 1.50]{cornea:2003} we obtain
$$  \secat(\Cuttil(\CC P^n)\to \Cut(\CC P^n)) \leq \frac{\mathrm{dim}(\Cut(\CC P^n))}{2} + 1 = 2n .    $$
Consequently by Theorem \ref{theorem_fibered_decomp} we obtain
$$  \GC(\CC P^n) \leq   \secat(\Cuttil(\CC P^n)\to \Cut(\CC P^n))  + 1 \leq 2n + 1   $$
for $n\geq 2$.
Since $\TC(\CC P^n) = 2n+ 1$ by \cite[Lemma 28.1]{farber:2006}, we obtain
$$   \GC(\CC P^n) = \TC(\CC P^n) = 2n + 1     $$
for $n\geq 2$.
The argument for $\mathbb{H}P^n$ is similar, using that $\mathbb{H}P^n$ is $3$-connected for all $n\geq 1$ and that $\TC(\mathbb{H}P^n)  = 2n+1$ by \cite[Corollary 3.16]{basabe:2014}.
Finally, for $n=1$ we have that $\CC P^1$ is isometric to $\mathbb{S}^2$ and $\mathbb{H}P^1$ is isometric to $\mathbb{S}^4$, where both $\mathbb{S}^2$ and $\mathbb{S}^4$ are equipped with the standard metric.
It is well-known that $\GC(\mathbb{S}^{2}) = \GC(\mathbb{S}^4) = 3$, see \cite[Proposition 4.1]{reciomitter:2021}, so this proves the assertion for $n = 1$.
\end{proof}

\section{Three-dimensional lens spaces} \label{sec_lens}

In this section we show that the total cut locus of a lens space of the form $L(p;1)$ with a metric of constant sectional curvature admits a fibered decomposition.
It is thus an example of a homogeneous Riemannian manifold which has this property, but which is not a globally symmetric space, see e.g. \cite[p. 105]{gilkey:2015}.
We will use the explicit fibered decomposition to derive an upper bound for the geodesic complexity of three-dimensional lens spaces of type $L(p;1)$.
We start by studying the cut locus of a point in the lens space $L(p;1)$, which was explicitly described by S. Anisov in \cite{anisov:2006}.
However, we give a self-contained exposition in this section, since we will need a detailed description of the tangent cut locus and of the cut locus in this setting.\medskip 

We consider the $3$-sphere as a subspace of $\CC^2$, i.e.
$$   \mathbb{S}^3 = \{ (z_1,z_2)\in\CC^2\,|\, z_1\overline{z}_1 + z_2\overline{z}_2 = 1\} .  $$
In the following we will also consider $\mathbb{S}^3$ as embedded in $\RR^4$ under the standard identification $\CC^2\cong\mathbb{R}^4$.
The special unitary group $SU(2)$ acts transitively on the $3$-sphere. 
Furthermore, for arbitrary $p\geq 3$, we have an action of $\ZZ_p$ on $\mathbb{S}^3$ denoted by $\Psi\colon \ZZ_p\times \mathbb{S}^3\to \mathbb{S}^3$, where $\ZZ_p$ is the cyclic group with $p$ elements, given by
\begin{equation} \label{eq_zp_action}
     \Psi(m, (z_1,z_2)) \mapsto  ( e^{\frac{2\pi i m }{p}} z_1 , e^{\frac{2 \pi im }{p}}z_2)   .   
\end{equation}
It is easy to see that this action is properly discontinuous.
If we equip $\mathbb{S}^3$ with the standard metric, then $\Psi$ is an action by isometries.
Consequently, we can equip the quotient $L(p;1) = \mathbb{S}^3/\ZZ_p$ with a metric for which $\pi\colon \mathbb{S}^3 \to L(p;1)$ becomes a Riemannian covering. We henceforth always consider $L(p;1)$ as equipped with such a metric.
The space $L(p;1)$ is called a \textit{lens space}.
Furthermore, note that the metric on $L(p;1)$ constructed in this way is a metric of constant sectional curvature. By the Killing-Hopf theorem all metrics of constant sectional curvature on $L(p;1)$ arise in this way, see e.g. \cite[Theorem 12.4 and Corollary 12.5]{lee:2018}.

Note that the action of $\ZZ_p$ on $\mathbb{S}^3$ commutes with the action of $SU(2)$.
Thus, $SU(2)$ acts on $L(p;1)$ and in particular this action is transitive, since it is already transitive on $\mathbb{S}^3$.
In the following we fix the point $p_0 = \pi(1,0)\in L(p;1)$. Its isotropy group under the $SU(2)$-action on $L(p;1)$ is easily seen to be
\begin{equation} \label{eq_lens_isotropy}
      K = \left\{ \begin{pmatrix} e^{\frac{2\pi i k}{p}} & 0 \\ 0 & e^{-\frac{2\pi i k}{p}} \end{pmatrix}\,\middle|\,k\in\{0,\ldots,p-1\} \right\} \cong \mathbb{Z}_p .  
\end{equation}
Note that for more general lens spaces of the form $L(p;q)$ where $p$ and $q$ are coprime with $q\neq 1$, see e.g. \cite[Example 2.43]{hatcher:2002}, the isometry group does not act transitively in general.
See \cite{kalliongis:2002} for details on the isometry groups of lens spaces.\medskip 


In order to describe the cut locus of a point $p_0\in L(p;1)$, let us first consider the more general situation of a Riemannian covering $\pi\colon\widetilde{M}\to M$. The following exposition closely follows \cite[Section 3]{ozols:1974}.

It is well known that geodesics in $\widetilde{M}$ are mapped to geodesics in $M$ under the Riemanian covering map $\pi$.
Assume that $M \cong \widetilde{M}/\Gamma$ where $\Gamma$ is a finite group of isometries of $\widetilde{M}$ acting properly discontinuously. Let $\mathrm{d}: \widetilde{M} \times \widetilde{M} \to \RR$ denote the distance function induced by the metric on $\widetilde{M}$. For any two distinct points $q,r\in \widetilde{M}$ we set
$$  H_{q,r} = \{u\in \widetilde{M}\,|\, \mathrm{d}(q,u) < \mathrm{d}(r,u)\}   .  $$
We recall from \cite[Definition 3.1]{ozols:1974} that
$$  \Delta_q = \bigcap_{g\in\Gamma\setminus\{e\}} H_{q,g\cdot q}  \subseteq \widetilde{M}     $$
is called the \textit{normal fundamental domain of} $\Gamma$ \textit{centered at }$q$. The following result by V. Ozols establishes a connection between normal fundamental domains and cut loci.

\begin{theorem}[{\cite[Corollary 3.11]{ozols:1974}}]
	Let $\pi: \widetilde{M} \to M$ be a Riemannian covering, let $q \in \widetilde{M}$ and let $\Delta_q \subset \widetilde{M}$ be its normal fundamental domain. If its closure satisfies $\overline{\Delta}_q \cap \Cut_{q}(\widetilde{M})=\emptyset$, then 
	$$\Cut_{\pi(q)}(M) = \pi(\partial \Delta_q).$$
\end{theorem}

Hence, to understand the cut locus of the point $\pi(q)\in M \cong \widetilde{M}/\Gamma$ we can study the boundary of the normal fundamental domain $\Delta_q$.
Let $\mathrm{inj}(T_q \widetilde{M})\subseteq T_q \widetilde{M}$ be the domain of injectivity of the exponential map in $\widetilde{ M}$ and put
$$  \widehat{\Delta}_q := (\exp_q|_{\mathrm{inj}(T_q\widetilde{M})})^{-1} (\overline{\Delta_q}) \subseteq T_q \widetilde{M} .$$
Assume in the following that $\overline{\Delta}_{q}\cap \Cut_q(\widetilde{M})=\emptyset$. Then $\exp_q$ maps $\widehat{\Delta}_q$ homeomorphically onto $\overline{\Delta}_q$, since the restriction of $\exp_{q}$ to $\mathrm{inj}(T_q\widetilde{M})$ is a homeomorphism onto its image. With  $K:= \mathrm{inj}(T_{\pi(q)} M)\cup \widetilde{\mathrm{Cut}}_{\pi(q)}(M)$ the diagram
$$  
\begin{tikzcd}
\widehat{\Delta}_{q}  \arrow[r, "D\pi_q", "\approx"'] \arrow[d, "\exp_q"', "\approx"]
& [2.5em]
K \arrow[]{d}{\exp_{\pi(q)} }
\\
\overline{\Delta}_{q} \arrow[]{r}{\pi} & M
\end{tikzcd}
$$
commutes and one checks that the maps
$   D\pi_q|_{\widehat{\Delta}_q}\colon \widehat{\Delta}_p \to K $ and $\exp_q|_{\widehat{\Delta}_q}\colon \widehat{\Delta}_p\to \overline{\Delta}_q  $
are homeomorphisms.
In particular, we see that $\partial\Delta_q$ is homeomorphic to the tangent cut locus $\Cuttil_{\pi(q)}(M)$ and that the exponential map
$$  \exp_{\pi(q)}|_{\widetilde{\mathrm{Cut}}_{\pi(q)}(M)} \colon \widetilde{\mathrm{Cut}}_{\pi(q)}(M) \to \mathrm{Cut}_{\pi(q)}(M) $$
can be understood by considering
$$  \pi|_{\partial\Delta_q} \colon \partial\Delta_q\to \mathrm{Cut}_{\pi(q)}(M) .     $$
In the following we denote by $\langle\cdot,\cdot\rangle$ the standard inner product on $\RR^4$.
The next lemma is easily shown by means of elementary geometry. Thus, we omit its proof.
\begin{lemma} \label{lemma_hqr}
Let $q,r\in\mathbb{S}^3$ be two distinct points.
Let $u = q-r\in \mathbb{R}^4$ and let $E_u$ be the $3$-plane of points in $\RR^4$ orthogonal to $u$.
Then 
$     H_{q,r} = \{ v\in\mathbb{S}^3 \,| \,  \langle v,u\rangle > 0\}  . $
\end{lemma}

We consider $\Delta_{q_0}$, the normal fundamental domain of $\ZZ_p$ centered at $q_0 = (1,0,0,0)\in\mathbb{S}^3$.
For $k\in \{0,\ldots,p-1\}$, we define 
$$   u_k = q_0 - k \cdot q_0 =  \big(1- \cos(\tfrac{2\pi  k }{p}), -\sin(\tfrac{2\pi  k }{p}),0,0\big) . $$
By Lemma \ref{lemma_hqr}, it is clear that
$$  \Delta_{q_0} = \{  r\in\mathbb{S}^3\,|\,  \langle u_k,r\rangle > 0\, \, \text{for} \,\, k= 1,\ldots,p-1\} .     $$
Its boundary is
$$ \partial\Delta_{q_0} = \left\{   r\in \mathbb{S}^3 \ \middle| \ \exists k \in \{1,\dots,p-1\} \text{ with } \left<u_k,r\right>=0, \left<u_{k'},r\right> \geq 0 \ \forall k' \in \{1,\dots,p-1\}\setminus \{k\}\right\}.$$
For $l \in\{1,\ldots,p-1\}$ and numbers $1\leq i_1<i_2<\ldots <i_l \leq p-1$, we define
$$ \widetilde{D}^{(l)}_{i_1,\ldots,i_l}= \left\{r \in \mathbb{S}^3 \ \middle| \ \langle u_{i_1},r\rangle=\dots=\langle u_{i_l},r\rangle =0, \ \langle u_j,r\rangle >0 \ \forall j \in \{1,\dots,p-1\}\setminus \{i_1,\dots,i_l\}\right\}. $$
It is clear that
$$  \partial\Delta_{q_0} =\bigsqcup_{\substack{   l \in\{1,\ldots,p-1\} \\ 1\leq i_1<\ldots <i_l \leq p-1 }}   \widetilde{D}^{(l)}_{i_1,\ldots,i_l}  .  $$
\begin{lemma} \label{lemma_lens_strat_cuttil}
All sets of the form $\widetilde{D}^{(l)}_{i_1,\ldots,i_l}$ are empty except $\widetilde{D}^{(1)}_1$, $\widetilde{D}^{(1)}_{p-1}$ and $\widetilde{D}^{(p-1)}_{1,\ldots,p-1}$.
Consequently, $\partial\Delta_{q_0}$ is the disjoint union of $\widetilde{D}^{(1)}_1$, $\widetilde{D}^{(1)}_{p-1}$ and $\widetilde{D}^{(p-1)}_{1,\ldots,p-1}$.
\end{lemma}
\begin{proof}
It is easy to see that
$$      \widetilde{D}^{(p-1)}_{1,\ldots,p-1} = \{ (0,0,x,y)\in\mathbb{S}^3\,|\, (x,y)\in\mathbb{S}^1\} .    $$
Hence, $\widetilde{D}^{(p-1)}_{1,\ldots,p-1}$ is non-empty.
For $l\in\{1,\ldots,p-1\}$, $l\neq \tfrac{p}{2}$, we set
$$\sigma_l = \frac{1-\cos(\tfrac{2\pi  l}{p})}{\sin(\tfrac{2\pi  l}{p})}.  $$
It can be seen directly that
$$\widetilde{D}^{(1)}_1 = \{ (a, \sigma_1 a,x,y)\in \mathbb{S}^3\,|\,a>0\}  \quad \text{and}\quad \widetilde{D}^{(1)}_{p-1} = \{ (a,\sigma_{p-1} a,x,y) \in \mathbb{S}^3\,|\, a>0\}   .    $$
Note that $\sigma_{p-1} = -\sigma_1$.
Let $m\in\{2,\ldots,p-2\}$.
We claim that $\widetilde{D}^{(1)}_{m} = \emptyset$.
Assume that there is a point $r = (a,b,x,y) \in \widetilde{D}^{(1)}_l$.
Then, $\langle r,u_m\rangle = 0$ implies that
$$   b = \sigma_m a \qquad \text{if} \,\, m\neq \frac{p}{2}.      $$
For arbitrary $m \in \{2,\ldots,p-2\}$, we get from $ \langle u_1 + u_{p-1} , r\rangle > 0$ that
\begin{equation} \label{eq_a_positive}
       2(1- \cos({\tfrac{2\pi}{p}})) a > 0    
\end{equation}
which implies that $a > 0$.
In case that $p$ is even and $m = \frac{p}{2}$, it can easily be seen that $a = 0$, yielding a contradiction to inequality \eqref{eq_a_positive}. Thus, $\widetilde{D}_{\frac{p}{2}}^{(1)}=\emptyset$.
Therefore, we assume throughout the rest of the proof that $m\neq \tfrac{p}{2}$.
We consider two separate cases, starting with $2\leq m < \tfrac{p}{2}$.
In this case we have $\sigma_m > 0$, so we see that $b > 0$.
We write $r =  (\widetilde{a} e^{i\varphi}, x+ i y)$ as an element of $\mathbb{C}^2$ with $\widetilde{a}>0$. 
It is clear that we have
$$    \tan \varphi = \frac{1-\cos(\frac{2 \pi m}{p})}{\sin(\tfrac{2\pi m}{p})}  > 0  $$
and we can choose $\varphi \in (0,\tfrac{\pi}{2})$. Since the third and fourth component of $u_1$ are trivial, we can use the Euclidean inner product  on $\RR^2$ to compute that
\begin{eqnarray*}
\langle u_1,r\rangle &=& 
\Bigg\langle \begin{pmatrix} 1- \cos(\tfrac{2\pi}{p})  \\  - \sin(\tfrac{2\pi }{p})    \end{pmatrix} , \begin{pmatrix}    \widetilde{a}\cos(\varphi) \\ \widetilde{a} \sin(\varphi)   \end{pmatrix} \Bigg\rangle_{\RR^2} \\  &=& \Bigg\langle \begin{pmatrix}   0 \\ -2 \sin(\tfrac{\pi}{p})      \end{pmatrix} , \widetilde{a} \begin{pmatrix}   \cos (\varphi - \tfrac{\pi}{p}) \\ \sin(\varphi - \tfrac{\pi}{p}) 
  \end{pmatrix}  \Bigg\rangle_{\RR^2} \\  &=&  -2\widetilde{a} \sin(\tfrac{\pi}{p})\sin(\varphi - \tfrac{\pi}{p}) ,
\end{eqnarray*}
where we rotated the vectors by an angle of $-\frac{\pi}{p}$ to get the second equality.
Note that by our assumption we have $\langle u_1,r \rangle > 0$ which implies
$  \sin(\tfrac{\pi}{p})\sin(\varphi-\tfrac{\pi}{p}) < 0 .$
Since $\varphi < \frac{\pi}{2}$ by assumption, we want to show that $\varphi >\frac{\pi}{p}$.  Then $\sin(\tfrac{\pi}{p})\sin(\varphi-\tfrac{\pi}{p})>0$, which is thus a contradiction.
The inequality $\varphi > \frac{\pi}{p}$ is equivalent to showing that $\tan(\varphi) > \tan(\frac{\pi}{p})$, i.e. that
\begin{equation} \label{eq_trigono} \frac{1 - \cos(\tfrac{2\pi m}{p})}{\sin(\tfrac{2 \pi m}{p})} \stackrel{!}{>} \frac{\sin(\tfrac{\pi}{p})}{\cos(\tfrac{\pi}{p})}.        \end{equation}

Note that since $m < \frac{p}{2}$, we have $\frac{\pi m}{p} < \frac{\pi}{2}$.
Consequently, 
$$     2 \cos(\tfrac{\pi m}{p})^2  < 2 \cos(\tfrac{\pi}{p})^2  .      $$
By standard trigonometry $$2 \cos(\tfrac{\pi m}{p} )^2= 1 + \cos(\tfrac{2 \pi m}{p})$$
and therefore
$$    1 - \cos( \tfrac{2\pi m}{p} )^2 <  2 \cos(\tfrac{\pi}{p})^2 ( 1- \cos(\tfrac{2 \pi m}{p} )    )   .      $$
One checks by direct computation that this is equivalent to
$$    (\sin(\tfrac{2 \pi m}{p}))^2 (\sin( \tfrac{\pi}{p}))^2 <   (1 - \cos(\tfrac{2 \pi m}{p} ))^2  (\cos(\tfrac{\pi}{p}))^2  .  $$
Since all squared terms were positive before squaring, we see that this is equivalent to
$$  \sin(\tfrac{2 \pi m}{p} ) \sin(\tfrac{\pi}{p} ) < (1 - \cos(\tfrac{2 \pi m}{p}  )) \cos(\tfrac{\pi}{p} )         $$
which clearly implies the inequality \eqref{eq_trigono}.
We thus get the desired contradiction in the case $2\leq m<\tfrac{p}{2}$.
The case $\tfrac{p}{2}< m\leq p-2$ can be treated similarly.
One can argue similarly that all sets of the form $\widetilde{D}^{(l)}_{i_1,\ldots,i_l}$ with $2\leq l\leq p-2$ are empty.
\end{proof}
To shorten our notation we write $\widetilde{D}^{(p-1)}$ for $\widetilde{D}^{(p-1)}_{1,\ldots,p-1}$.
Set $p_0 = \pi(q_0)\in L(p;1)$ and recall that $$D\pi_{q_0}\circ (\exp_{q_0}|_{\partial\widehat{\Delta}_{q_0}})^{-1} \colon \partial\Delta_{q_0} \to \Cuttil_{p_0}(L(p;1))  $$
is a homeomorphism. Moreover, the diagram
$$  
\begin{tikzcd}
\partial \widehat{\Delta}_{q_0} \arrow[]{r}{D\pi_{q_0}} \arrow[swap]{d}{\exp_{q_0}} & [2em] \Cuttil_{p_0}(L(p;1))  \arrow[]{d}{\exp_{p_0}}
\\
\partial\Delta_{q_0} \arrow[]{r}{\pi} & \Cut_{p_0}(L(p;1))
\end{tikzcd}
$$
commutes. Here, we obviously consider the restrictions of the maps to the spaces occurring in the diagram, which we drop from the notation for the sake of readability.
We denote the images of the sets $\widetilde{D}^{(1)}_i$ by
$$   \widetilde{C}^{(1)}_i =  (D\pi_{q_0}\circ (\exp_{q_0}|_{\partial\widehat{\Delta}_{q_0}})^{-1}) (\widetilde{D}_i^{(1)}),\quad \text{for}\,\,\,i\in \{1,p-1\}  $$
and similarly
$$   \widetilde{C}^{(p-1)} = (D\pi_{q_0}\circ (\exp_{q_0}|_{\partial\widehat{\Delta}_{q_0}})^{-1} ) (\widetilde{D}^{(p-1)}) .   $$

\begin{prop} \label{prop_lens_cut}
Let $p \in \NN$ with $p \geq 3$ and consider the lens space $L(p;1)$ with a Riemannian metric of constant sectional curvature. Let $\pi\colon \mathbb{S}^3\to L(p;1)$ be the corresponding Riemannian covering and put $p_0:=\pi(1,0,0,0)$.
\begin{enumerate}
    \item  The sets $\widetilde{C}^{(1)}_1,\widetilde{C}^{(1)}_{p-1}$ and $\widetilde{C}^{(p-1)}$ form a locally compact decomposition of the tangent cut locus $\widetilde{\mathrm{Cut}}_{p_0}(L(p;1))$.
Moreover, $\widetilde{C}^{(1)}_1$ and $\widetilde{C}^{(1)}_{p-1}$ are homeomorphic to open $2$-disks and $\widetilde{C}^{(p-1)}$ is homeomorphic to $\mathbb{S}^1$.
\item The cut locus $\mathrm{Cut}_{p_0}(L(p;1))$ admits a locally compact decomposition into
$$   C^{(1)} = \pi(\widetilde{D}^{(1)}_i)  = \exp_{p_0}(\widetilde{C}^{(1)}_{i}), \,\,i\in \{1,p-1\}  $$ and $$ C^{(p-1)} = \pi(\widetilde{D}^{(p-1)}) = \exp_{p_0}(\widetilde{C}^{p-1})  .     $$
The map $\exp_{p_0}|_{\widetilde{C}^{(1)}_{i}} \colon \widetilde{C}^{(1)}_i\to C^{(1)}$ is a homeomorphism for $i\in \{1,p-1\}$.
Under suitable identifications of $\widetilde{C}^{(p-1)}$ and $C^{(p-1)}$ with $\mathbb{S}^1$, the map $$\exp_{p_0}|_{\widetilde{C}^{(p-1)}}\colon\widetilde{C}^{(p-1)} \to C^{(p-1)} $$ can be identified with the standard $p$-fold covering of $\mathbb{S}^1$ by $\mathbb{S}^1$.
\end{enumerate}
Hence, $$\Cut_{p_0}(L(p;1))=C^{(1)}\sqcup C^{(p-1)}$$
is a fibered decomposition of $\Cut_{p_0}(L(p;1))$ and the associated fibrations are fiber bundles.
\end{prop}
\begin{proof}
The first part is apparent given the identification $\partial\Delta_{q_0} \approx \widetilde{\mathrm{Cut}}_{p_0}(L(p;1))$ and the characterization of the sets $\widetilde{D}^{(1)}_1,\widetilde{D}^{(1)}_{p-1}$ and $\widetilde{D}^{(p-1)}$ in the proof of Lemma \ref{lemma_lens_strat_cuttil}.
For the second part, we note that
$$   \Psi(1, \widetilde{D}^{(1)}_{p-1}) = \widetilde{D}^{(1)}_1 .    $$
Consequently, $\widetilde{D}^{(1)}_1$ and $\widetilde{D}^{(1)}_{p-1}$ are identified under $\pi$.
Furthermore, the restriction of $\pi$ to $\widetilde{D}^{(1)}_{i}$, $i\in\{1,p-1\}$ is a homeomorphism onto its image since it is continuous, injective and a local homeomorphism.
The same properties therefore hold for $\widetilde{C}^{(1)}_1$, $\widetilde{C}^{(1)}_{p-1}$ and the map $\exp_{p_0}$ under the identification $\partial\Delta_{q_0} \cong \widetilde{\mathrm{Cut}}_{p_0}(L(p;1))$.
Recall that 
$$ \widetilde{D}^{(p-1)}    = \{ (0,z)\in\mathbb{S}^3\,|\, z\in\mathbb{S}^1\} ,  $$
thus it is obviously homeomorphic to $\mathbb{S}^1$ and the $\ZZ_p$-action on $\mathbb{S}^3$ becomes the standard $\ZZ_p$-action on $\mathbb{S}^1 $ under this identification.
Since the map $\mathbb{S}^1\to  \mathbb{S}^1/\ZZ_p$ is a $p$-fold covering, this proves the last claim.
\end{proof}

In the following, we want to show that the fibered decomposition of $\Cut_{p_0}(L(p;1))$ is isotropy-invariant with respect to the transitive $SU(2)$-action to obtain a fibered decomposition of the total cut locus of $L(p;1)$ from Theorem \ref{theorem_isotropy_fibered}.


\begin{lemma} \label{lemma_number_of_minimal_geo}
Let $(M,g)$ be a Riemannian manifold and let $q\in M$ be a point.
Furthermore, let $m\geq 2$ be an integer.
Assume that $G$ is a group of isometries of $M$ which fixes $q$.
Let $S_m\subseteq \Cut_q(M)$ be the set of points $r\in \Cut_q(M)$ such that there are precisely $m$ distinct minimal geodesics between $q$ and $r$.
Then $S_m$ is invariant under $G$.
\end{lemma}
\begin{proof}
Let $\varphi\colon G\times M\to M$ denote the $G$-action and let $\Phi\colon G\times GM\to GM$ denote the induced pointwise $G$-action, given by
$$   \Phi(g,\gamma)(t) =    \Phi_g(\gamma)(t) = \varphi(g,\gamma(t)), \quad \text{for}\,\,\,t\in[0,1],g\in G,\gamma\in GM .   $$
Let $r \in S_m$ and $g \in G$. Since $r$ is a cut point, $ s = \varphi(g,r)\in\Cut_q(M)$.
Let $\gamma_1,\ldots,\gamma_m$ be the $m$ distinct minimal geodesics between $q$ and $r$.
Then $\Phi_g(\gamma_1),\ldots,\Phi_g(\gamma_m)$ are distinct minimal geodesics between $q$ and $s$.
If there was a minimal geodesic $\sigma$ between $q$ and $s$ which is distinct from all $\Phi_g(\gamma_i)$, $i \in \{1,\ldots, m\}$,
then $\Phi_{g^{-1}}(\sigma)$ would be a minimal geodesic joining $q$ and $r$ distinct from $\gamma_1,\ldots,\gamma_m$. This contradicts $r\in S_m$, hence such a $\sigma$ does not exist and we derive that $s \in S_m$ as well. This proves the claim.
\end{proof}

\begin{cor} \label{cor_lens_decomp}
The fibered decomposition of $\Cut_{p_0}(L(p;1))= C^{(1)}\sqcup C^{(p-1)}$ constructed in Proposition \ref{prop_lens_cut} is isotropy-invariant.
\end{cor}
\begin{proof}
By Proposition \ref{prop_lens_cut} we can characterize $C^{(1)}$ and $C^{(p-1)}$ as 
\begin{align*}
   C^{(1)} &= \{ q\in \Cut_{p_0}(L(p;1))\,|\, \text{there are precisely two minimal geodesics joining } p_0 \text{ and } q \} , \\
   C^{(p-1)} &= \{ q\in \Cut_{p_0}(L(p;1))\,|\, \text{there are precisely } p \text{ minimal geodesics joining } p_0 \text{ and } q \}  .  
\end{align*}
Therefore the isotropy invariance is a direct consequence by Lemma \ref{lemma_number_of_minimal_geo}.
\end{proof}

It follows from Theorem \ref{theorem_isotropy_fibered} and Corollary \ref{cor_lens_decomp} that there is a decomposition of $\Cut(L(p;1))$ into sets $A^{(1)}$ and $A^{(p-1)}$ which form a fibered decomposition of $\Cut(L(p;1))$.
We now want to study this decomposition in greater detail.

Recall that we denote the isotropy group of the $SU(2)$-action on $L(p;1)$ by $K$ and computed it in equation \eqref{eq_lens_isotropy}.
In order to better distinguish the various group actions, let
$$  \Phi\colon SU(2)\times \mathbb{S}^3 \to \mathbb{S}^3 \quad \text{and}\quad \varphi\colon SU(2)\times L(p;1)\to L(p;1)    $$
be the actions of $SU(2)$ on $\mathbb{S}^3$ and on $L(p;1)$, respectively. 
Recall that we denoted the $\ZZ_p$-action on $\mathbb{S}^3$ be $\Psi$, see equation \eqref{eq_zp_action}.
If $A\in SU(2)$ we shall also write $\Phi_A$ for the diffeomorphism $\Phi(A,\cdot)\colon \mathbb{S}^3 \to \mathbb{S}^3$ and similarly for the other actions.

The fibered decomposition of $\Cut(L(p;1))$ is given as follows.
For $l\in\{1,p-1\}$, we have
$$  A^{(l)} = \{ (q,r)\in \Cut(L(p;1))\,|\, r\in \varphi_A(C^{(l)}), A\in SU(2)  \,\,\text{such that}\,\,\mathrm{pr}(A) = q \}   , $$
where $\mathrm{pr}\colon SU(2)\to L(p;1)$ is the canonical projection.
We denote the preimages of $A^{(1)}$ and $A^{(p-1)}$ in the total tangent cut locus by $\widetilde{A}^{(1)}$ and $\widetilde{A}^{(p-1)}$.
Explicitly, we have
$$  \widetilde{A}^{(1)} = \{ (q,v)\in\Cuttil(L(p;1))\,|\,  v\in (D \varphi_A)_{p_0} (\widetilde{C}^{(1)}_1\cup \widetilde{C}^{(1)}_{p-1})    , \,\,A\in SU(2)\,\,\text{such that}\,\,\mathrm{pr}(A) = q \}          $$ and similarly for $\widetilde{A}^{(p-1)}$.
By Proposition \ref{prop_lens_cut}, Corollary \ref{cor_lens_decomp} and Theorem \ref{theorem_isotropy_fibered} we obtain that $\widetilde{A}^{(1)}\to A^{(1)}$ is a $2$-fold covering and that $\widetilde{A}^{(p-1)}\to A^{(p-1)}$ is a $p$-fold covering, where we allow coverings to be trivial, i.e. the total space of the covering might not be connected.
We want to show that $\widetilde{A}^{(1)}$ consists of two connected components which implies that $\widetilde{A}^{(1)}\to A^{(1)}$ is a trivial covering.

\begin{lemma}
The set $\widetilde{C}^{(1)}_1\subseteq T_{p_0}L(p;1)$ is isotropy-invariant with respect to the induced $SU(2)$-action in the tangent bundle $TL(p;1)$. More precisely if $A\in K$, then $(D\varphi_A)_{p_0}(\widetilde{C}^{(1)}_1) = \widetilde{C}^{(1)}_1$.
The same holds for $\widetilde{C}^{(1)}_{p-1}$.
\end{lemma}
\begin{proof}
Let $x\in \widetilde{C}^{(1)}_1\subseteq T_{p_0}L(p;1)$ and $A\in K$, i.e. there is a $k\in\{0,\ldots,p-1\}$ such that 
$$  A = \begin{pmatrix} e^{\frac{2\pi i k}{p}} & 0 \\ 0 & e^{-\frac{2\pi i k }{p}}   \end{pmatrix} .     $$
It holds that $(\Psi_{-k}\circ\Phi_A)(q_0) = q_0$.
We want to show that $(D\varphi_A)_{p_0}(x)\in \widetilde{C}^{(1)}_1$. Consider the following diagram.
$$
\begin{tikzcd}
\mathbb{S}^3 \arrow[]{r}{\Phi_A} & [4em] \mathbb{S}^3 \arrow[]{r}{\Psi_{-k}} & [4em] \mathbb{S}^3 \\
T_{q_0}\mathbb{S}^3 \arrow[]{u}{\exp_{q_0}} \arrow[]{r}{(D\Phi_A)_{q_0}} \arrow[swap]{d}{D\pi_{q_0}} 
&
T_{\Phi_A(q_0)}\mathbb{S}^3 \arrow[]{u}{\exp_{\Phi_A(q_0)}} \arrow[]{r}{(D\Psi_{-k})_{\Phi_A(q_0)}}   \arrow[swap]{d}{D\pi_{\Phi_A(q_0)}}  
& T_{q_0}\mathbb{S}^3 \arrow[swap]{u}{\exp_{q_0}} \arrow[]{d}{D\pi_{q_0}}
\\
T_{p_0}L(p;1) \arrow[swap]{r}{(D\varphi_A)_{p_0}} & T_{p_0}L(p;1) \arrow[swap]{r}{\id} & T_{p_0}L(p;1)
\end{tikzcd}
$$
The lower two squares commute by definition of $\varphi$ and the fact that the induced action of $\Psi$ on $L(p;1)$ is trivial.
The upper two squares commute by the naturality of the exponential map.
Note that all arrows in the lower two squares are isomorphisms.

If we restrict to $\Cuttil_{p_0}(L(p;1))$ and to $\partial\Delta_{q_0}$, respectively, we obtain a commutative diagram
$$
\begin{tikzcd}
\partial\Delta_{q_0} \arrow[]{r}{ \Psi_{-k}\circ\Phi_A} \arrow[swap]{d}{D\pi_{q_0}\circ \exp_{q_0}^{-1}} 
& [4em]
\partial\Delta_{q_0} \arrow[]{d}{D\pi_{q_0}\circ \exp_{q_0}^{-1}}
\\
\Cuttil_{p_0}(L(p;1)) \arrow[]{r}{(D\varphi_A)_{p_0}} & \Cuttil_{p_0}(L(p;1)).
\end{tikzcd}
$$

By the proof of Lemma \ref{lemma_lens_strat_cuttil} we can write $ y = (\exp_{q_0}\circ(D\pi)_{q_0}^{-1})(x)\in \widetilde{D}^{(1)}_1$ as
$$ y  = ((1 + i \sigma_1) a, z) \quad \text{where} \,\,\, a> 0, \,\, z\in\CC,   $$
and where $\sigma_1$ was defined in the proof of Lemma \ref{lemma_lens_strat_cuttil}.
Then it follows that 
$$   (\Psi_{-k}\circ\Phi_A) (  y) = \Big( (1 + i \sigma_1) a, e^{-\frac{4 \pi i k}{p}}  z \Big),  $$
which is again an element of $\widetilde{D}^{(1)}_1$.
Consequently, $(D\varphi_A)_{p_0}(x)\in \widetilde{C}^{(1)}_1$.
The argument for $\widetilde{C}^{(1)}_{p-1}$ is analogous.
\end{proof}

By the previous lemma, the sets
$$  \widetilde{A}^{(1)}_1 = \{ (q,v)\in\Cuttil(L(p;1))\,|\,  v\in (D \varphi_A)_q( \widetilde{C}^{(1)}_1)    , \,\, A\in SU(2)\,\,\text{such that}\,\,\mathrm{pr}(A) = q \}          $$
and 
$$  \widetilde{A}^{(1)}_{p-1} = \{ (q,v)\in\Cuttil(L(p;1))\,|\,  v\in (D \varphi_A)_q (\widetilde{C}^{(1)}_{p-1})    , \,\,A\in SU(2)\,\,\text{such that}\,\,\mathrm{pr}(A) = q \}          $$
are well-defined.
Moreover, we clearly have $\widetilde{A}^{(1)} = \widetilde{A}^{(1)}_{1} \sqcup \widetilde{A}^{(1)}_{p-1} $. 
Since $\widetilde{A}^{(1)}\to A^{(1)}$ is a fiber bundle by Theorem \ref{theorem_isotropy_fibered}, we now see that $\widetilde{A}^{(1)}\to A^{(1)}$ is a trivial $2$-fold covering.
This implies that
$$  \secat(\widetilde{A}^{(1)}\to A^{(1)}) = 1 .    $$

\begin{theorem} \label{theorem_lens}
Let $p \in \NN$ with  $p \geq 3$ and consider the lens space $L(p;1)$ with a metric of constant sectional curvature.
Then
$$   6 \leq \GC(L(p;1))\leq 7 .    $$
\end{theorem}
\begin{proof}
M. Farber and M. Grant have shown in \cite[Corollary 15]{farber:2008grant} that the topological complexity of $L(p;1)$ is $\TC(L(p;1)) = 6$, which yields $\GC(L(p;1))\geq \TC(L(p;1))=6$, see \cite[Remark 1.9]{reciomitter:2021}.
By Theorem \ref{theorem_fibered_decomp} we have
\begin{equation} \label{eq_up_lens}
      \GC(L(p;1)) \leq \secat(\widetilde{A}^{(1)}\to A^{(1)}) + \secat(\widetilde{A}^{(p-1)}\to A^{(p-1)}) + 1 .   
\end{equation}
As argued above, we have $ \secat(\widetilde{A}^{(1)}\to A^{(1)}) = 1$.
Recall that $A^{(p-1)}$ is a circle bundle over $L(p;1)$, therefore it is $4$-dimensional and we get
$$   \secat(\widetilde{A}^{(p-1)}\to A^{(p-1)}) \leq \cat(A^{(p-1)}) \leq 5     $$
by \cite[Theorem 18]{schwarz:1966} and \cite[Theorem 1.50]{cornea:2003}.
Thus, the inequality \eqref{eq_up_lens} gives the upper bound $\GC(L(p;1))\leq 7$.
\end{proof}

\bibliography{lit}

\providecommand{\bysame}{\leavevmode\hbox to3em{\hrulefill}\thinspace}
\providecommand{\MR}{\relax\ifhmode\unskip\space\fi MR }
\providecommand{\MRhref}[2]{%
  \href{http://www.ams.org/mathscinet-getitem?mr=#1}{#2}
}
\providecommand{\href}[2]{#2}
\begin{thebibliography}{BGRT14}

\bibitem[Ani06]{anisov:2006}
Sergei Anisov, \emph{Cut loci in lens manifolds}, C. R. Math. Acad. Sci. Paris
  \textbf{342} (2006), no.~8, 595--600.

\bibitem[BCV18]{blaszczyk:2018}
Zbigniew B{\l}aszczyk and Jos\'{e}~Gabriel Carrasquel-Vera, \emph{Topological
  complexity and efficiency of motion planning algorithms}, Rev. Mat. Iberoam.
  \textbf{34} (2018), no.~4, 1679--1684.

\bibitem[Bes78]{besse:1978}
Arthur~L. Besse, \emph{Manifolds all of whose geodesics are closed}, Ergebnisse
  der Mathematik und ihrer Grenzgebiete, vol.~93, Springer-Verlag, Berlin-New
  York, 1978.

\bibitem[BGRT14]{basabe:2014}
Ibai Basabe, Jes\'{u}s Gonz\'{a}lez, Yuli~B. Rudyak, and Dai Tamaki,
  \emph{Higher topological complexity and its symmetrization}, Algebr. Geom.
  Topol. \textbf{14} (2014), no.~4, 2103--2124.

\bibitem[Bre13]{bredon:2013}
Glen~E Bredon, \emph{Topology and geometry}, Graduate Texts in Mathematics,
  vol. 139, Springer Science \& Business Media, 2013.

\bibitem[BtD95]{broecker:85}
Theodor Br\"{o}cker and Tammo tom Dieck, \emph{Representations of compact {L}ie
  groups}, Graduate Texts in Mathematics, vol.~98, Springer-Verlag, New York,
  1995.

\bibitem[CLOT03]{cornea:2003}
Octav Cornea, Gregory Lupton, John Oprea, and Daniel Tanr\'e,
  \emph{Lusternik-{S}chnirelmann category}, Mathematical Surveys and
  Monographs, vol. 103, American Mathematical Society, Providence, RI, 2003.

\bibitem[Far03]{farber:2003}
Michael Farber, \emph{Topological complexity of motion planning}, Discrete
  Comput. Geom. \textbf{29} (2003), no.~2, 211--221.

\bibitem[Far06]{farber:2006}
\bysame, \emph{Topology of robot motion planning}, Morse theoretic methods in
  nonlinear analysis and in symplectic topology, NATO Sci. Ser. II Math. Phys.
  Chem., vol. 217, Springer, Dordrecht, 2006, pp.~185--230.

\bibitem[Far08]{farber:2008}
\bysame, \emph{Invitation to topological robotics}, Zurich Lectures in Advanced
  Mathematics, European Mathematical Society (EMS), Z\"urich, 2008.

\bibitem[FG08]{farber:2008grant}
Michael Farber and Mark Grant, \emph{Robot motion planning, weights of
  cohomology classes, and cohomology operations}, Proceedings of the American
  Mathematical Society \textbf{136} (2008), no.~9, 3339--3349.

\bibitem[GC19]{garcia:2019}
Jose~Manuel Garc{\'\i}a-Calcines, \emph{A note on covers defining relative and
  sectional categories}, Topology Appl. \textbf{265} (2019), 106810, 14.

\bibitem[GPVL15]{gilkey:2015}
Peter Gilkey, JeongHyeong Park, and Ram\'{o}n V\'{a}zquez-Lorenzo,
  \emph{Aspects of differential geometry. {II}}, Synthesis Lectures on
  Mathematics and Statistics, vol.~16, Morgan \& Claypool Publishers,
  Williston, VT, 2015.

\bibitem[Hat02]{hatcher:2002}
Allen Hatcher, \emph{Algebraic topology}, Cambridge University Press, 2002.

\bibitem[Hel78]{helgason:78}
Sigurdur Helgason, \emph{Differential geometry, {L}ie groups, and symmetric
  spaces}, Pure and Applied Mathematics, vol.~80, Academic Press, Inc.
  [Harcourt Brace Jovanovich, Publishers], New York-London, 1978.

\bibitem[KM02]{kalliongis:2002}
John Kalliongis and Andy Miller, \emph{Geometric group actions on lens spaces},
  Kyungpook Math. J. \textbf{42} (2002), no.~2, 313--344.

\bibitem[Lee18]{lee:2018}
John~M. Lee, \emph{Introduction to {R}iemannian manifolds}, second ed.,
  Graduate Texts in Mathematics, vol. 176, Springer, Cham, 2018.

\bibitem[MS21]{mescher:2021}
Stephan Mescher and Maximilian Stegemeyer, \emph{Geodesic complexity of
  homogeneous {R}iemannian manifolds}, arXiv:2105.09215, to appear in: Algebr.
  Geom. Topol., 2021.

\bibitem[Ozo74]{ozols:1974}
Vilnis Ozols, \emph{Cut loci in {R}iemannian manifolds}, Tohoku Math. J. (2)
  \textbf{26} (1974), 219--227.

\bibitem[RM21]{reciomitter:2021}
David Recio-Mitter, \emph{Geodesic complexity of motion planning}, J. Appl. and
  Comput. Topology \textbf{5} (2021), 141--178.

\bibitem[Sak77]{Sakai3}
Takashi Sakai, \emph{On cut loci of compact symmetric spaces}, Hokkaido Math.
  J. \textbf{6} (1977), no.~1, 136--161.

\bibitem[Sak78a]{Sakai2}
\bysame, \emph{Cut loci of compact symmetric spaces}, Minimal submanifolds and
  geodesics ({P}roc. {J}apan-{U}nited {S}tates {S}em., {T}okyo, 1978),
  North-Holland, Amsterdam-New York, 1978, pp.~193--207.

\bibitem[Sak78b]{Sakai1}
\bysame, \emph{On the structure of cut loci in compact {R}iemannian symmetric
  spaces}, Math. Ann. \textbf{235} (1978), no.~2, 129--148.

\bibitem[Sch66]{schwarz:1966}
Albert~S. Schwarz, \emph{The genus of a fiber space}, Amer. Math. Soc. Transl.
  \textbf{55} (1966), 49--140.

\bibitem[Ste51]{steenrod:2016}
Norman Steenrod, \emph{The topology of fibre bundles}, Princeton Mathematical
  Series, vol. 14, Princeton University Press, Princeton, N. J., 1951.

\end{thebibliography}
 \bibliographystyle{amsalpha}
 
\end{document}